\documentclass[11pt]{amsart}
\usepackage{hyperref}
\usepackage{tikz}
\usepackage{tikz-cd}
\usetikzlibrary{matrix,arrows,decorations.pathmorphing}
\usepackage{comment}

\usepackage{amsfonts,latexsym,amsthm,amssymb,amsmath,amscd,euscript,graphicx}
\usepackage[super]{nth}
\usepackage{setspace}
\usepackage{framed}
\usepackage{verbatim}
\usepackage{calc,color,relsize}
\usepackage{graphicx,mathabx}
\usepackage{float,epsfig}
\usepackage{enumerate} 
\usepackage{verbatim}
 \usepackage[foot]{amsaddr}

\newtheorem{theorem}{Theorem}
\newtheorem{corollary}{Corollary}
\newtheorem*{corollary*}{Corollary}
\newtheorem*{theorem*}{Theorem}

\newtheorem*{conjecture*}{Conjecture}
\newtheorem{proposition}[theorem]{Proposition}
\newtheorem*{proposition*}{Proposition}

\newtheorem*{example*}{Example}
\newtheorem{question}{Question}
\newtheorem{lemma}[theorem]{Lemma}
\newtheorem*{lemma*}{Lemma}
\newtheorem{definition}{Definition}
\def\Z{\mathbb Z}

\def\V{\mathcal{V}}
\def\W{\mathcal{W}}
\def\X{\mathcal{X}}
\def\Y{\mathcal{Y}}
\def\Z{\mathcal{Z}}
\def\E{\mathcal{E}}
\def\MM{\mathbb{M}}
\DeclareMathOperator{\Tr}{Tr}
\DeclareMathOperator{\Hess}{Hess}

\newcommand{\todo}[1]{\textbf{\textcolor{red}{[#1]}}}

\title[K\"ahler-Ricci Flow and Anti-Bisectional Curvature]{K\"ahler-Ricci Flow preserves negative anti-bisectional curvature}
\author[Iowa State University]{Gabriel Khan}
\email{gkhan@iastate.edu}
\author[Chongqing Normal University]{Fangyang Zheng}
\email{franciszheng@yahoo.com}

\begin{document}

\maketitle

\begin{abstract}
   In recent work (Pure Appl. Anal. 2 (2020), 397-426 \cite{KhanZhang}), the first named author and J. Zhang found a connection between the regularity theory of optimal transport and the curvature of K\"ahler manifolds. In particular, we showed that the Ma-Trudinger-Wang (MTW) tensor for a cost function $c(x,y)=\Psi(x-y)$ can be understood as the anti-bisectional curvature of an associated K\"ahler metric defined on a tube domain. Here, the anti-bisectional curvature is defined as $R(\X, \overline \Y,\X,\overline \Y) $ where $\X$ and $\Y$ are polarized $(1,0)$ vectors and $R$ is the curvature tensor. The correspondence between the anti-bisectional curvature and the MTW tensor provides a meaningful sense in which the anti-bisectional curvature can have a sign (i.e., be positive or negative).
    
    In this paper, we study the behavior of the anti-bisectional curvature under K\"ahler-Ricci flow. Somewhat unexpectedly, we find that \textit{non-positive} anti-bisectional curvature is preserved under the flow. In complex dimension two, we also show that \textit{non-negative} orthogonal anti-bisectional curvature (i.e., the MTW(0) condition) is preserved under the flow. We provide several applications of these results --- in complex geometry, optimal transport, and affine geometry.

\end{abstract}

\section{Introduction}

For the past four decades, Hamilton's Ricci flow \cite{Hamiltonthreemanifolds} has played a central role in geometric analysis. Starting with an initial Riemannian manifold $(M,g)$, this flow evolves the metric according to its Ricci curvature:
\begin{equation}\label{Ricciflow}
    \frac{\partial g}{\partial t}  = - 2 \, \textrm{Ric}(g).
\end{equation}
This acts as a reaction-diffusion equation for the curvature and can be used to deform a Riemannian metric towards a canonical geometry. 

When the underlying manifold is K\"ahler, the Ricci flow preserves the complex structure and the K\"ahlerity of the metric \cite{CaoDeformation}. Furthermore, it is possible to write the flow in terms of the K\"ahler potentials $\Psi(z,t)$ in a holomorphic coordinate chart $\{z^i \}_{i=1}^n$. Doing so, the potential evolves according to the parabolic complex Monge-Amp\'ere equation 
\begin{equation} \label{KahlerRicciflow}
\frac{\partial \Psi}{\partial t}=2\log \left( \operatorname{det}\left [ \frac{\partial^{2} \Psi}{\partial z^{i} \partial \bar{z}^{j}}\right] \right).
\end{equation}
For K\"ahler manifolds, this allows us to study the Ricci flow, which is originally a weakly parabolic system of equations, in terms of a single fully nonlinear parabolic equation.

In this paper, we study the K\"ahler-Ricci flow on tube domains (i.e. domains of the form $(\Omega + \sqrt{-1} \mathbb{R}^n) \subset \mathbb{C}^n$ where $\Omega$ is a convex domain in $\mathbb{R}^n)$, or more generally on the tangent bundles of Hessian manifolds.
Our focus is on the behavior of the ``anti-bisectional curvature," which is defined as
\begin{equation}
    \mathfrak{A}(\V,\W)= R(\V,\overline{\W},\V,\overline{\W})
\end{equation}
for $(1,0)$-vectors $\V$ and $\W$.

By considering metrics which are translation symmetric in the imaginary directions and restricting our attention to polarized holomorphic vectors, there is a meaningful sense in which the anti-bisectional curvature can be positive or negative. Here, we say that a $(1,0)$-vector $\V$ is \textit{polarized} if $ \V = v^i \frac{\partial}{\partial z^i}, $
where the coefficients $v^i$ are real numbers that depend only on the point $x \in \Omega$ (in the holomorphic coordinates $ \{ z^i = x^i + \sqrt{-1} y^i \}_{i=1}^n$ with $x \in \Omega$ and  $y \in \mathbb{R}^n$). For a more precise definition of polarized vectors and signed anti-bisectional curvature, see Subsection \ref{Anti-bisectionalcurvature}. We find that \textit{negativity} of the anti-bisectional curvature is preserved under K\"ahler-Ricci flow. With the exception of negative Gaussian curvature for Riemann surfaces, every other curvature condition which was previously known to be preserved under Ricci flow is a positivity condition, so this result is somewhat unexpected. 

\begin{theorem} \label{Main theorem}
Let $(T \Omega, \omega_0)$ be a complete K\"ahler manifold with bounded curvature and non-positive anti-bisectional curvature. Consider the K\"ahler-Ricci flow $(T \Omega,\omega_t)$ with initial metric $\omega_0$. For all time that the flow exists, the anti-bisectional curvature remains non-positive.
\end{theorem}

We also study the orthogonal anti-bisectional curvature (i.e., the antibisectional curvature restricted to orthogonal polarized vectors) and show that in complex dimension two, \textit{positivity} of this quantity is also preserved under the flow.

\begin{theorem}\label{SurfaceswithNOAB}
Suppose that $ \Omega \subset \mathbb{R}^2 $ is a convex domain and $\Psi :\Omega \to \mathbb{R}$ is a strongly convex function so that the associated K\"ahler manifold $(T \Omega, \omega_0)$ 
\begin{enumerate}
    \item is complete,
    \item has bounded curvature, and
    \item has non-negative orthogonal anti-bisectional curvature.
\end{enumerate}
 Then the orthogonal anti-bisectional curvature remains non-negative along K\"ahler-Ricci flow. 
\end{theorem}


\subsection{Applications in complex geometry}

For tube domains, Theorem \ref{Main theorem} provides a new way to study the relationship between the holomorphic sectional curvature and Ricci curvature, which is an important question in complex geometry. Recently, Wu and Yau \cite{WuYau} showed that on any K\"ahler manifold $(\mathbb{M},\omega_0)$ with negatively pinched holomorphic sectional curvature, there exists a K\"ahler-Einstein metric of negative scalar curvature which is uniformly equivalent to the original metric $\omega_0$. However, the following question remains open.

\begin{question}
When is it the case that a K\"ahler manifold admits a K\"ahler-Einstein metric with negative holomorphic sectional curvature?
\end{question}


For general K\"ahler manifolds, this is a very difficult problem. Establishing the existence of a K\"ahler-Einstein metric is already a non-trivial question (see, e.g., \cite{YauCalabi, YauRicci}), and the existence theorems do not provide control over the sign of the holomorphic sectional curvature. 

Although negative anti-bisectional curvature does not imply that the holomorphic sectional curvature is negative for all $(1,0)$-vectors, it does imply that the holomorphic sectional curvature of polarized vectors remains non-positive under the flow. Combining Theorem \ref{Main theorem} with work of Tong studying the convergence of negatively pinched metrics \cite{Tong}, we show the following result.

\begin{corollary} \label{NegativeantibisectionalKahlerEinstein}
Suppose that $T \Omega$ admits a complete K\"ahler metric $\omega_0$ whose holomorphic sectional curvature is negatively pinched and whose anti-bisectional curvature is non-positive. Then the complete K\"ahler-Einstein metric on $T \Omega$ with scalar curvature $-2$ has non-positive anti-bisectional curvature.
\end{corollary}

When $\Omega$ is a bounded and strongly convex domain, it is possible to sharpen this estimate and show that the polarized holomorphic sectional curvature of the K\"ahler-Einstein metric is negatively pinched.

\begin{corollary} \label{NegativecostcurvedKahlerEinstein}
Suppose that $\Omega$ is a smooth, bounded, and strongly convex domain and that $(T \Omega, \omega_0)$ is a complete K\"ahler metric with negative anti-bisectional curvature and negatively pinched holomorphic sectional curvature. Then the complete K\"ahler-Einstein metric on $T \Omega$ has negative cost-curvature. 
\end{corollary}


Negative cost-curvature is a stronger version of non-positive anti-bisectional curvature. The etymology of this definition comes from optimal transport, and is described in Subsection \ref{Optimaltransportbackground}.

\begin{definition}[Negative cost-curvature]\label{Negative cost-curvature}
A metric on a tube domain is negatively cost-curved if its anti-bisectional curvature is non-positive and its holomorphic sectional curvature satisfies $H(\X) \leq K \|\X\|^4$ for some $K < 0$ whenever $\X$ is a polarized holomorphic vector.
\end{definition}

For compact Hessian manifolds, it is possible to prove a version of Corollary \ref{NegativecostcurvedKahlerEinstein} under the assumption that the first affine Chern class is negative, without assuming that the holomorphic sectional curvature of the tangent bundle is negatively pinched (see Definition 6 \cite{PuechmorelTo} for a definition of affine Chern classes).

\begin{corollary}\label{CompactHessianEinstein}
Let $(M, g_0, D)$ be a compact Hessian manifold. Suppose that the associated K\"ahler metric $\omega_0$ on $T M$ has non-positive anti-bisectional curvature and that the first affine Chern class of $M$ is negative. Then, under the normalized K\"ahler-Ricci flow, $T M$ converges to a K\"ahler-Einstein metric which is negatively cost-curved.
\end{corollary}



\subsubsection{Instantaneous control of the metric}
Even without the assumption that the holomorphic sectional curvature of $T \Omega$ is negatively pinched, negative anti-bisectional curvature provides strong control over the geometry along the flow. In particular, to bound the curvature of the positive time metrics when the initial metric has negative anti-bisectional curvature, it is sufficient to bound a single quantity $\mathfrak{O}$, which we call the orthogonal anti-bisectional trace curvature (see Definition \ref{OABTC} for a formal definition).

\begin{corollary}\label{uniformcontrol}
Suppose that $(T \Omega, \omega_0)$ is a K\"ahler metric  whose scalar curvature is bounded below by a constant $-K$ and whose anti-bisectional curvature is non-positive.
Consider the positive time metric  $(T \Omega, \omega_{t_0})$
and $\E = \{ \E_i \}_{i=1}^n$ be a polarized unitary frame at time $t_0$ and $\X$ be a polarized holomorphic unit vector with respect to $\omega_{t_0}$. Then there exists a constant $C(n)$ depending only on the dimension $n$ so that along K\"ahler-Ricci flow, the polarized holomorphic sectional curvature satisfies the estimate
\begin{equation} \label{Uniformcontrolinequality}
    0 \geq H(\X) \geq - C(n) \left( \frac{n}{t+ \frac{n}{K}}+ 2 \mathfrak{O}(\E) \right). 
\end{equation} 

\end{corollary}

Note that for any $K \geq 0$, we have the uniform bound
\begin{equation}  - \frac{C(n)}{t+\frac{n}{K}} <  - \frac{C(n)}{t}. \end{equation} 

As such, the previous corollary provides a bound on the polarized holomorphic sectional curvature which is \textit{independent} of the initial scalar curvature and only depends on $\mathfrak{O}(\E)$ at time $t_0$. It is possible to show that for metrics with negative anti-bisectional curvature, the curvature tensor can be bounded from above and below in terms of the polarized holomorphic sectional curvature and the orthogonal anti-bisectional trace curvature (see Appendix \ref{PHSCandO} for details). Thus, this corollary shows that in order to control the geometry of the metric at a time $t_0>0$, we need only control $\mathfrak{O}$.

To prove Corollary \ref{uniformcontrol}, we establish a quantitative version of Berger's formula for the scalar curvature \cite{Berger}. This estimate has other applications in complex geometry, so we mention it now.

\begin{proposition}
Suppose $\mathbb{M}^n$ is a K\"ahler manifold with non-positive holomorphic sectional curvature. Then there exists a constant $C_n$ depending only on the dimension so that for all unit $(1,0)$ tangent vectors $\X$, the scalar curvature $S$ satisfies the inequality
\begin{equation}
S \leq C_n H(\X). \end{equation}

For metrics with non-negative holomorphic sectional curvature, we obtain the same result with the final inequality reversed.
\end{proposition}

\subsection{Surfaces with non-negative orthogonal anti-bisectional curvature and optimal transport}

Apart from its applications in complex geometry, anti-bisectional curvature plays an important role in optimal transport. In particular, given a K\"ahler metric on a tube domain whose orthogonal anti-bisectional curvature is \textit{non-negative}, we can define an associated cost function which satisfies the MTW(0) condition (see Subsection \ref{Optimaltransportbackground} for more details). This condition plays an important role in the regularity theory of optimal transport, in that it precludes any local obstructions to the smoothness of transport maps.

For this reason, it is worthwhile to determine whether the non-negativity of orthogonal anti-bisectional curvature is preserved under the flow. Theorem \ref{SurfaceswithNOAB} shows that in complex dimension two (but in no other dimension), this is indeed the case. Due to the correspondence between the MTW tensor and the orthogonal anti-bisectional curvature, this result allows us to generalize the MTW(0) condition for cost functions of the form $c(x,y) = \Psi(x-y)$ where $\Psi$ is a convex function on an open domain in $\mathbb{R}^2$.

\begin{definition}[KR weak regularity] \label{KRWeakregularity}
A cost function $c(x,y)=\Psi_0(x-y)$ is KR weakly regular if there exists a solution to the parabolic flow
\begin{equation} \label{Psicostflow}
  \begin{cases}  \frac{ \partial} {\partial t} \Psi(x,t) = 2 \log \left( \det[ \Hess \Psi(x,t)]  \right)\\
 \Psi(x,0) = \Psi_0(x) \end{cases} 
\end{equation} 
 which has non-negative orthogonal anti-bisectional curvature for positive time.
\end{definition}

Using this flow, we prove an interior H\"older estimate for optimal transport which only relies on a $W^{2,p}$ estimate for $\Psi_0$ (see Corollary \ref{Weak continuity} for a precise statement). 
At present, this result is limited in that to define the above flow we must consider complete K\"ahler metrics. One natural question is whether these results can be extended to incomplete K\"ahler metrics, which often appear in optimal transport.


\begin{question}
Can we find suitable boundary conditions so that the above flow is well defined and Theorem \ref{SurfaceswithNOAB} holds when $\Psi_0$ does not induce a complete metric?
\end{question}

\subsection{Examples}

At first, one might worry that negative anti-bisectional curvature is an overly strong assumption and that the results of the previous sections may be vacuous. To show that this is not the case, in Section \ref{examples}, we provide several example of metrics which satisfy this assumption.
Of particular interest is the following example.

\begin{example*}
Let $\mathbb{B}$ be the unit ball in $\mathbb{R}^n$ and consider the tube domain $T \mathbb{B}$. Let $\omega_0$ be the K\"ahler metric induced by the potential
\begin{equation} \Psi(z) = -\log \left( 1-\sum_{i=1}^n x_i^2 \right),\end{equation} 
where $z^i = x^i +\sqrt{-1} y^i$ and $\Psi$ is independent of the fiber directions $y$. 
\end{example*}

This is a negative cost curved metric on $T \mathbb{B}$ with negative holomorphic sectional curvature. Appealing to Corollary \ref{NegativecostcurvedKahlerEinstein}, this shows that the K\"ahler-Einstein metric on $T \mathbb{B}$ (which was originally studied by Calabi \cite{Calabi}) also has negative cost-curvature.

\section{Background}

We now provide background material on Hessian manifolds, anti-bisectional curvature, optimal transport, and K\"ahler-Ricci flow. This section is not intended to be a complete reference; we focus on results that are needed in this paper. However, we have included references which provide a more thorough background for the interested reader.

\subsection{Background on K\"ahler-Sasaki metrics}
\label{KahlerSasakibackground}

In this paper, we will primarily study tube domains. However, our results can be generalized to the tangent bundles of compact or parallelizable Hessian manifolds (we reserve $\Omega$ for convex domains, so denote Hessian manifolds by $M$). We refer to such metrics as \textit{K\"ahler-Sasaki metrics}. For a Riemannian manifold $(M,g)$ with an affine connection $D$, the \textit{Sasaki metrics} is an associated almost-Hermitian structures on the tangent bundle $T M$  (see \cite{Dombrowski, Satoh} for a more complete reference). In this paper, we restrict our attention to the case where $(M,g,D)$ is a Hessian manifold, in which case the Sasaki metric will be K\"ahler, hence the term \textit{K\"ahler-Sasaki}.

\begin{definition}[Hessian manifolds] \label{Hessianmanifolds}
A Riemannian manifold $(M, g)$ with an affine connection $D$ is said to be a Hessian manifold if
\begin{enumerate}
    \item $D$ is a flat (curvature- and torsion-free) connection,
    \item  such that around every point $x \in M$, there is an open neighborbood $U_x \subset M$,
    \item and a function $\Psi: U_x \to \mathbb{R}$, such that
    \begin{equation}\label{Hessian potential}
        g = D^2 \, \Psi. 
    \end{equation} 
\end{enumerate}
\end{definition}

At first, it might seem unnecessary to consider convex domains as Hessian manifolds. However, there is one conception advantage to using this approach; by considering a convex domain $\Omega$ as a Hessian manifold, we can change coordinates in an affine way. Doing so will play an important role in the proof of Theorem \ref{Main theorem}, so we take some time to discuss this more general setting.

To define a K\"ahler metric on the tangent bundle of a Hessian manifold, we observe that the connection $D$ induces an atlas $\mathcal{A}$ of coordinates charts whose transition maps are affine functions. We consider such a chart (denoted $x$) and construct coordinates on the tangent bundle $T M$ by $(x,y)$, where the fiber coordinates $y$ are defined as $y^j(\mathcal{X}) = \mathcal{X}^j$ for a vector $\mathcal{X}=\mathcal{X}^i \partial_{x^i}$. Finally, we consider the complex coordinates $z = x + \sqrt{-1} y$. This coordinate chart is defined on $T U_x$. Furthermore, it is possible to construct such coordinates in a neighborhood of every point in $M$. Since the transition maps of the charts in $\mathcal{A}$ are affine, the transition maps of their associated charts on $T M$ are holomorphic. As such, the tangent bundle of an affine manifold is a complex manifold.

To define a K\"ahler metric on the tangent bundle, we lift the Hessian potential $\Psi$ to the tangent bundle $T M$ using the \textit{horizontal lift}, defined as
\begin{equation} \label{horizontal lift}
    \Psi^h(z) = \Psi(x)
\end{equation} 
for $z = x + \sqrt{-1} y$.
 It is worth noting that a K\"ahler-Sasaki metric  is complete if and only if the underlying Hessian manifold $(M,g,D)$ is \cite{Molitor}. 

Given a convex domain $\Omega$ in Euclidean space, it is possible to construct a Hessian metric simply by choosing a convex potential function and using the connection induced by differentiation in coordinates. 
In this case, the tangent bundle of $\Omega$ is simply the tube domain $T \Omega$. Furthermore, the K\"ahler metric has a natural $\mathbb{R}^n$ symmetry induced by translations in the $y$-coordinate.



 \subsubsection{Lifting vectors}
 \label{liftingvectors}
It is possible to lift tangent vectors of $\Omega$ to tangent vectors of the tube domain $T \Omega$. Conceptually, this is quite simple (at least for tube domains), but it is essential to the notion of signed anti-bisectional curvature. As such, we take some time to define the concept carefully. Before we do so, we remark on some notational conventions.

To discuss lifting vectors, we must consider the tangent bundle of $\Omega$ separately from the tube domain $T \Omega$. To denote the tangent bundle of $\Omega$, we use the notation $T^{real} \Omega$. We will generally denote elements of $T^{real} \Omega$ with lower case letters (e.g. ``$v$'') and tangent vectors on $ T \Omega$ (i.e., elements of $T T \Omega$) with uppercase calligraphic letters (e.g. ``$\V$''). When it is clear from context, we drop the notation for which base point we are using.
In other words, $v \in T^{real} \Omega$ is used as shorthand for $v \in T^{real}_x \Omega$. Furthermore, in order to avoid over-using the letter $T$, we will often simply use the letter $\mathbb{M}$ to refer to K\"ahler manifolds, K\"ahler-Sasaki or otherwise.

It is possible to lift vectors from $T^{real} \Omega$ (or $T^{real} M$ for a general Hessian metric) to $T \mathbb{M}$ in several different ways. In this paper, we will only consider a single lift, which takes a tangent vector 
\begin{equation}
v =\sum_i  v^i \frac{\partial}{\partial {x^i}} \in  T^{real}_{x_0} \Omega \end{equation} to the $(1,0)$ vector 
\begin{equation} \V = v^{pol.} := \sum_i v^i \frac{\partial}{\partial {z^i}} \in T_{(x_0,y)}^{(1,0)} \mathbb{M}.
\end{equation}
As before, $\{ z^i\}_{i=1}^n$ are the holomorphic coordinates $x^i + \sqrt{-1} y^i$. In this formula, we are free to pick the fiber coordinate $y$ at will. Since the metric is translation symmetric in the fibers, this choice will not affect the geometry. We say that the lifted vectors are polarized, since their components are restricted to an $n$-dimensional subspace of $T \mathbb{M}$. 
Furthermore, we can use the space of polarized (1,0) vectors to make sense of the notion of positive or negative anti-bisectional curvature.

\subsection{Anti-bisectional curvature}
\label{Anti-bisectionalcurvature}
 Given two $(1,0)$ vectors $\V$ and $\W$, the anti-bisectional curvature is defined to be
\begin{equation}
     \mathfrak{A}(\V,\W)= R(\V,\overline \W, \V, \overline \W).
\end{equation}

At first, it does not seem meaningful to say that the anti-bisectional curvature is either positive or negative. While the quantity $R(\V, \overline \W, \V, \overline \W)$ is well defined for any two $(1,0)$-vectors $\V$ and $\W$, if we simply consider the vectors $\sqrt{-1} \V$ and $\W$, we have $R(\V, \overline \V, \V, \overline W) = -R(\sqrt{-1} \V, \overline \W, \sqrt{-1} \V, \overline \W)$.

However, for K\"ahler-Sasaki metrics, by restricting our attention to polarized vectors, it is meaningful to talk about this quantity being positive or negative.

\begin{definition}[Non-positive anti-bisectional curvature] \label{antibisectional}
The K\"ahler metric $\mathfrak{h}_\Psi$ has non-positive orthogonal anti-bisectional curvature, if $\mathfrak{A} (\V,\W)\leq 0$ for any $\V= v^{pol.}$, $\W=w^{pol.}$ with $v,w \in T^{real} \Omega$. 
\end{definition}

By taking $v = w$, we can see that
 negative anti-bisectional curvature implies negative holomorphic sectional curvature for vectors which are \textit{polarized}. However, negative anti-bisectional curvature does not imply negative holomorphic sectional curvature for arbitrary holomorphic vectors. It also does not imply negative Ricci curvature (see Subsection \ref{AntibisectionalandRicci} for some consequences of this).

\begin{definition}[Non-negative orthogonal anti-bisectional curvature] \label{NOAB}
The K\"ahler metric $\mathfrak{h}_\Psi$ has non-negative orthogonal anti-bisectional curvature (abbreviated (NOAB)), if $\mathfrak{A} (\V,\W)\geq 0$ for any $\V=v^{pol.}$, $\W = w^{pol.}$ with $v,w \in T^{real} \Omega$ such that $g(v,w)= 0$.
\end{definition}

We also define the orthogonal anti-bisectional trace curvature to be the trace of the orthogonal anti-bisectional curvature.

\begin{definition}[Orthogonal anti-bisectional trace curvature] \label{OABTC}
Let $\{ \E_i \} \subset T_z T \Omega $ be a unitary frame of polarized holomorphic vectors. The  orthogonal anti-bisectional trace curvature is defined to be the quantity
\begin{equation}
    \mathfrak{O} = \sum_{i \neq j} R(\E_i, \overline{ \E_j}, \E_i, \overline{ \E_j}) . 
\end{equation}

\end{definition}

The orthogonal anti-bisectional trace curvature generally depends on the choice of polarized unitary frame (which corresponds to an orthonormal frame of $T^{real} \Omega$). As such, we can interpret the orthogonal anti-bisectional trace curvature as a map from the orthonormal frame bundle of $\Omega$ to $\mathbb{R}$. In this paper, we will often fix such a frame and consider this as a scalar function.
 Non-negative orthogonal anti-bisectional curvature does not imply any sign for the holomorphic sectional curvature or the scalar curvature.


\subsubsection{The curvature of tube domains}

Heuristically, we expect invariant metrics on tube domains to be negatively curved. This heuristic can be made precise in various ways. For instance, when $\Omega$ is a smooth strongly convex domain that does not contain any line, $T \Omega$ is biholomorphic to a smooth bounded strongly pseudoconvex domain. As such, by a result of Klembeck \cite{Klembeck}, it admits a complete metrics of strongly negative holomorphic sectional curvature. Furthermore, we can use results of Cheng and Yau \cite{ChengYau} (or alternatively the recent work of Wu and Yau \cite{WuYau}) to show that such tube domains admit K\"ahler-Einstein metrics with negative scalar curvature.

With these results in mind, it is reasonable to try to find complete metrics of negative anti-bisectional curvature (in that $R(\V,\overline \W ,\V ,\overline W)<0$ for all $\V=v^{pol.}, \W=w^{pol.}$) or negative cost-curvature. 
However, when the complex dimension is greater than one, we are not aware of any examples with \textit{strongly negative} anti-bisectional curvature  (in that \begin{equation} R(\V,\overline \W ,\V ,\overline W) < -\kappa |V|^2 |W|^2 \end{equation} for some $\kappa>0$). In particular, any complex space form has vanishing orthogonal anti-bisectional curvature. Furthermore, for a smooth and strictly pseudo-convex tube domain, the K\"ahler-Einstein metric converges to a complex hyperbolic metric near its boundary, which implies that its orthogonal anti-bisectional curvature attenuates. As such, negative cost-curvature is the correct way to strengthen the hypothesis of negative anti-bisectional curvature so that the assumption does not become vacuous.

\subsection{Optimal transport and the role of anti-bisectional curvature}

\label{Optimaltransportbackground}

Optimal transport studies the most economical way to transport resources. This problem was originally studied in 1781 \cite{Monge} and combines aspects from analysis, probability and geometry. For two comprehensive references on the subject, we refer the reader \cite{Santambrogio} and \cite{OTON}.
An understanding of this subject is not necessary to understand the main results of this paper, so this subsection may be skipped by readers whose primary interest is complex geometry. However, the regularity theory of optimal transport provides the original motivation for studying anti-bisectional curvature and we will use several results about optimal transport in Subsection \ref{NOABApplications}.

The Monge problem of optimal transport is the following.
 We are given an initial probability space $(X, \mu)$, a target probability space $(Y,\nu)$ and a cost function $c: X \times Y \to \mathbb{R}$, which provides the cost of transporting a unit of mass from a point $x \in X$ to a point $y \in Y$.
We then seek to find a $\mu$-measurable map (i.e., a transport plan) $T:X \to Y$ such that such that the total cost is minimized.
 \begin{equation} \label{Mongeproblem}
      \int_X c(x,T(x)) d \mu = \inf_{S_\sharp \mu=\nu} \int_X c(x,S(x)) d \mu 
 \end{equation}
Here, the notation $T_\sharp \mu = \nu$ means that
$\mu[T^{-1}(Y)] = \nu[Y]$ for each Borel set $Y \subset{\mathbb{R}^n}$.

This is a special case of the Monge-Kantorovich problem of optimal transport, which tries to minimize the integral
\begin{equation} \label{Kantorovichproblem}
\min _{\pi \in \Pi\left(\mu, \nu \right)} \int_{X \times Y} c(x, y) d \pi(x, y),
\end{equation}
where $\Pi\left(\mu, \nu \right)$ denotes the space of all couplings between $\mu$ and $\nu$. The Monge-Kantorovich problem can be solved in great generality and has a dual formulation as a linear program. We will make use of this generalization in Corollary \ref{Weak continuity}, so we will mention it now mention it now.

Under fairly general conditions, the solution to the Monge-Kantorovich problem will be unique and also be a solution to the Monge problem. For instance, this is the case whenever
\begin{enumerate}
    \item  $X$ and $Y$ are domains of $\mathbb{R}^n$,
    \item  $\mu$ and $\nu$ are absolutely continuous with respect to the Lebesgue measure, and
    \item $c$ satisfies some smoothness and non-degeneracy assumptions.
\end{enumerate}
When $c$ and the measures are sufficiently smooth, the transport map $T$ is given by the $c$-subdifferential of a potential function  $u: X \to \mathbb{R}$, which weakly solves the Monge-Amp\`ere type equation:
\begin{equation} \label{MongeAmpere}
\det(\nabla^2 u + A(x, \nabla u)) = B(x,\nabla u). 
\end{equation} 

Here, $A$ is a matrix-valued function which depends on the cost function and $B$ is a scalar valued function which depends on the cost function and the two probability measures. Furthermore, the boundary conditions force the image of the $c$-subdifferential of $u$ on $X$ to be $Y$. For a more complete overview on the existence theory of the Monge problem, we refer the reader to Chapter 10 of Villani's text \cite{OTON} (see also Brenier \cite{YB} and Gangbo-McCann \cite{GOT} for the original results).

Once the existence of an optimal transport map has been established, it is natural to ask is whether the transport map is continuous, in that nearby points in initial configuration $(X, \mu)$ get sent to nearby points in the target $(Y,\nu)$. This problem was first studied in the special case $c(x,y) = |x-y|^2/2$ (see, e.g., \cite{LC} \cite{Delanoe} \cite{Urbas}).

For general cost functions, the interior regularity of $u$ was studied in a landmark 2005 paper of Ma, Trudinger and Wang \cite{MTW}, and the global regularity was addressed by Trudinger and Wang several years later \cite{TW}. In particular, Ma, Trudinger and Wang showed that for smooth cost functions, in order to obtain a $C^2$-estimate on $u$ (and thus a Lipschitz estimate on $T$), it is necessary to make the following assumption.

\begin{definition}[MTW(0) Condition]\label{MTW0definition}
Let $\xi \in \mathbb{R}^n$ be a vector and $\eta \in (\mathbb{R}^n)^\ast$ be a covector with $ \eta(\xi)=0$. We say a cost function satisfies the MTW(0) condition if the following inequality holds
\begin{equation}\label{MTW0}
\mathfrak{S}(\xi,\eta) := \sum_{i,j,k,l,p,q,r,s} (c_{ij,p}c^{p,q}c_{q,rs}-c_{ij,rs})c^{r,k}c^{s,l} \xi^i \xi^j \eta^k \eta^l \geq 0.
\end{equation}

Here, we use the notation $c_{I,J} = \frac{\partial^{|I|+|J|}}{ \partial x^I \partial y^J}$
for any pair of multi-indices $I,J$ and use the notation $c^{i,j}$ to denote the matrix inverse of the mixed Hessian $c_{i,j}$.
\end{definition}

The quantity $\mathfrak{S}$ is known as the MTW tensor and plays a central role in the regularity theory of optimal transport. More specifically, this assumption prevents local obstructions for smoothness of the transport map. Loeper  found a geometric perspective for this condition and showed that for costs which do not satisfy MTW(0), it is possible to find smooth measures\footnote{There is also a global obstruction to regularity, which is avoided whenever the spaces $X$ and $Y$ are \textit{relatively $c$-convex} (\cite{MTW} Definition 2.3). We will use this assumption in Corollary \ref{Weak continuity}, but will not need it anywhere else in this paper.} for which the solution to the Monge optimal transport problem is \textit{discontinuous}
\cite{Loeper}. As such, cost functions which satisfy the MTW(0) condition are said to be \textit{weakly regular}. Furthermore, cost functions which satisfy the inequality while dropping the orthogonality of $ \eta$ and $\xi$ are said to have non-negative cost-sectional curvature. 
For costs of the form $c(x,y)= \Psi(x-y)$ (henceforth \textit{$\Psi$-costs}), the relationship between the $\mathfrak{S}$ and the anti-bisectional curvature was observed by the first author and Zhang.

\begin{theorem*}[\cite{KhanZhang} Theorem 6]
Let $c$ be a cost function of the form $c(x,y)= \Psi(x-y)$ for some convex function $\Psi: \Omega \to \mathbb{R}$. 
Then the MTW tensor of $c$ is proportional to the anti-bisectional curvature of the tube domain $T \Omega$ whose K\"ahler potential is $\Psi^h$.
\end{theorem*}

In view of this result, we see that Theorem \ref{SurfaceswithNOAB} is relevant to optimal transport. In particular, for KR weakly regular two-variable $\Psi$-costs whose convex potential induces a complete Hessian metric, we can flow in time to find new costs which are weakly regular.

\subsubsection{Parabolic Flows and Optimal Transport}
The solution to the Monge problem involves solving an elliptic equation, so most of the research has focused on elliptic equations. However, there have been several works which have studied parabolic flows in the context of optimal transport. In particular, Kitagawa introduced a parabolic flow of potentials which converges to the solution of a Monge problem \cite{KitagawaParabolic}. For the cost function $c(x,y) = \frac{1}{2}\|x-y\|^2$ (better known as the squared-distance cost), the highest order term of this flow corresponds to K\"ahler-Ricci flow if we lift the potentials to obtain a flow on the tube domain $T \Omega$. There are extra drift and potential terms to ensure that the flow converges to the Monge solution and non-linear boundary conditions that serve to preserve the mass. However, the formal correspondence to K\"ahler-Ricci flow is very suggestive. 

There is one important conceptual difference between the flow in Kitagawa's work and the one we consider here. In the former, the potential evolves to converge to the solution of an optimal transport problem. In this paper, the flow instead deforms the \emph{cost function}.

Using Ricci flow to evolve the cost function has also been studied \cite{McCannTopping, Topping}. However, these works take a somewhat different perspective in that the cost function is either the square-distance cost $c(x,y) = d(x,y)^2/2$ or Perelman's reduced length functional $\mathcal{L}$ \cite{Perelmanentropy}. Further more, the Ricci flow is used to evolve the Riemannian metric, so the evolution of the cost function is implicit. In this paper, we instead consider costs which are induced by a convex potential and the flow directly evolves the cost function.

 From an optimal transport perspective, this might seem unusual at first. However, there are reasons to think that this deformation is useful. First, K\"ahler-Ricci flow tends to make the geometry of a space more homogeneous. As such, if we consider a two dimensional $\Psi$-cost which is KR weakly regular but has points where the MTW tensor is strongly positive, after a short period of time we can expect the MTW tensor to be strictly positive. 
Furthermore, the Ricci potential (i.e., $\log \left( \det[ \Hess \Psi] \right))$ has an optimal transport interpretation, in that it gives a quantitative measure for how strongly the cost $c(x,y)= \Psi(x-y)$ is twisted (see \cite{OTON} Chapter 12 for a in-depth discussion of twisting for cost functions).


\subsubsection{Completeness of $\Psi$-costs}

At present, there is a limitation in this theory, in that K\"ahler-Ricci flow is only uniquely defined for $\Psi$-costs where $(\mathbb{M},\omega_\Psi)$ is complete as a K\"ahler manifold. We refer to such $\Psi$-costs as \emph{complete}. Completeness is a natural assumption in K\"ahler geometry. However, it is not as important in optimal transport. As such, there are many $\Psi$-costs which are weakly regular but not complete. Notable among these is the cost 
\begin{equation} \label{logcost}
    c(x,y)=\log \left(1 + \sum_{i=1}^n \exp(x_i-y_i)\right).
\end{equation}
This cost function appears in mathematical finance \cite{PalWong}. From a geometric perspective, it induces a metric of constant positive holomorphic sectional curvature which is not complete. This implies that the cost \eqref{logcost} is weakly regular and has non-negative cost curvature. However, since the associated K\"ahler metric is not complete, there are many possible Ricci flows with this initial condition, many of which will not preserve the weak regularity. It is of interest to find appropriate boundary conditions for incomplete metrics so that Theorem \ref{SurfaceswithNOAB} continues to hold.

\subsection{Background on K\"ahler-Ricci flow}
\label{BackgroundonKRflow}

We now provide some background on K\"ahler-Ricci flow. In this section, we will focus only on the results that we will need for this paper. For a much more complete overview, we refer to the book by Boucksom et. al. \cite{Boucksom}.
For complete K\"ahler manifolds, the flow is given by the parabolic flow
\begin{equation}\label{KRflowwithinitialconditions}
 \begin{cases} 
\frac{\partial}{\partial t} \omega(t) =-2\operatorname{Ric}(\omega(t))-\lambda \omega(t) \\
\omega(0) =\omega_{0}.
\end{cases} 
\end{equation} 

Here, $\lambda$ is a parameter which serves to normalize the flow. When $\lambda$ is zero, this is the unnormalized K\"ahler-Ricci flow. We will also consider the case $\lambda = 2$, which we refer to as the normalized K\"ahler-Ricci flow.
Our main focus in this paper are curvature conditions which are preserved under the flow.

\begin{definition}[Preserved curvature conditions]
A curvature condition $\kappa$ is preserved under K\"ahler-Ricci flow if whenever $\omega_0$ satisfies $\kappa$, the metrics $\omega_t$ also satisfy $\kappa$ for $t>0$. 
\end{definition}


 For compact manifolds, Hamilton proved a tensor maximum principle which states that a curvature condition is preserved under the flow whenever it satisfies a null-vector condition \cite{HamiltonFourmanifolds}.
As such, to find invariant curvature conditions we must understand how the curvature evolves under the K\"ahler-Ricci flow. To do so, we consider holomorphic coordinates $\{z^i \}_{i=1}^n$ and write out the curvature tensor as
 \begin{equation}
     R_{i \overline j k \overline \ell}  =  R\left( \frac{\partial}{ \partial z^i}, \frac{\partial}{ \partial \overline z^j},\frac{\partial}{ \partial z^k},\frac{\partial}{ \partial \overline z^\ell} \right).
 \end{equation} 
 
 Computing the time-evolution of this tensor, we find that
\begin{eqnarray}
\frac{ \partial} {\partial t} R_{i \overline j k \overline \ell} & = & \Delta R_{i \overline j k \overline \ell} - \lambda R_{i \overline j k \overline \ell} \nonumber \\
& &+ 2 R_{i \overline j p \overline r} R_{s \overline q k \overline \ell} h^{\overline p  q} h^{r \overline s} + 2 R_{i \overline \ell p \overline r} R_{s \overline q k \overline j} h^{\overline p  q} h^{r \overline s}-  2 R_{i \overline q k \overline r} R_{p \overline j s \overline \ell} h^{\overline p  q} h^{r \overline s} \label{TermsinQ} \\
& & - \left( \textrm{Ric}_{i \overline q} R_{p \overline j k \overline \ell} h^{\overline p  q} + \textrm{Ric}_{p \overline j} R_{i \overline q k \overline \ell} h^{\overline p  q}+ \textrm{Ric}_{k \overline q} R_{i \overline j p \overline \ell} h^{\overline p  q}+R_{p \overline \ell} \textrm{Ric}_{i \overline j k \overline q} h^{\overline p  q} \right) \label{Uhlenbeckterms} .
\end{eqnarray}

From this, we can see that K\"ahler-Ricci flow behaves like a non-linear reaction-diffusion equation for the curvature. The Laplacian term acts to diffuse the curvature throughout the space and the rest of the terms are reaction terms.

The term involving $\lambda$ comes from the renormalization of the metric and vanishes when $\lambda = 0$. If we consider curvature conditions which are scale-invariant, this term will play no role, so we ignore it in our analysis. From a conceptual perspective, the terms in line \eqref{Uhlenbeckterms} are an artifact of the coordinates, and do not come from any intrinsic geometric evolution. In other words, it is possible to find a gauge transformation which makes these terms vanish (see Equation \eqref{Uhlenbecktrick} for the precise gauge transformation). 

Thus, the crucial reaction terms are those in line \eqref{TermsinQ}. Hamilton's maximum principle states than to find an invariant curvature tensors, we can discard the diffusion terms and study the system of the ordinary differential equations
\begin{equation} \label{curvatureODE}
    \frac{ \partial} {\partial t} R_{i \overline j k \overline \ell} = 2 R_{i \overline j p \overline r} R_{s \overline q k \overline \ell} h^{\overline p  q} h^{r \overline s} + 2 R_{i \overline \ell p \overline r} R_{s \overline q k \overline j} h^{\overline p  q} h^{r \overline s}-  2 R_{i \overline q k \overline r} R_{p \overline j s \overline \ell} h^{\overline p  q} h^{r \overline s},
\end{equation} 
which defines a flow of the space of algebraic curvature tensors. To show that a curvature condition is preserved, we must only verify that these reaction terms are ``inward-pointing" for any algebraic curvature tensor on the boundary of the curvature condition.

The terms in Equation \eqref{curvatureODE} are quadratic in the curvature. Heuristically, they also tend to be positive (see \cite{Wilking} for a Lie-algebraic perspective on this phenomena). For this reason, the curvature conditions which are preserved are generally positivity conditions (see \cite{ Brendle, Cao, GuZhang, HamiltonFourmanifolds, Mok, Nguyen} for examples). In fact, the only previously known example of an invariant negativity condition was negative curvature for Riemann surfaces. For this reason, our main result is somewhat unexpected.



Before moving on, it is worth noting that we are primarily concerned with complete K\"ahler manifolds. As such, we will not be able to use Hamilton's maximum principle directly. Instead, we will use a version which applies to complete manifolds and was originally proven by Shi \cite{Shi}.



\subsubsection{Analytic Preliminaries}
\label{Analyticpreliminaries}
In order to study the behavior of K\"ahler-Ricci flow on tube domains, we must first know that solutions exist and are unique. Fortunately, this problem has been studied extensively in the literature and we will be able to use existing results to establish existence, uniqueness and convergence.

When Hamilton introduced the Ricci flow for compact manifolds, his proof of the existence and uniqueness used the Nash-Moser inverse function theorem. Deturck later found a much simpler proof by conjugating Ricci flow by a time dependent diffeomorphism \cite{Deturck}. However, these results concern compact manifolds, and so do not directly apply to our setting.

For complete manifolds, it is no longer possible to directly use the standard maximum principle, which makes establishing existence and uniqueness much more difficult. The pioneering work for complete non-compact spaces was done by Shi, who established the following result.

\begin{theorem*}[\cite{Shi} Theorems 2.1 and 5.1]
Let $(\mathbb{M}, \omega_0)$ be a complete K\"ahler manifolds with bounded curvature. The K\"ahler-Ricci flow exists for some positive time $T$. Furthermore, it preserves the complex structure and the metrics remain K\"ahler under the flow.
\end{theorem*}

For a tube domain with a complete K\"ahler metric of bounded curvature, this result shows that the K\"ahler-Ricci flow exists. However, a priori one might worry that if one starts with a K\"ahler-Sasaki metric, the solution might fail to be K\"ahler-Sasaki after some time. To address this concern, we can use a result of Chen and Zhu which shows that for metrics of bounded curvature, the flow is unique and that it preserves the isometry group.

\begin{theorem*}[\cite{ChenZhu} Corollary 1.2]
Suppose $(M^n,g_{ij}(x) )$ is a complete Riemannian manifold and suppose $g_{ij}(x,t)$ is a solution to the Ricci flow with bounded curvature on $M^n \times [0,T]$ and with $g_{ij}(x)$ as initial data. If $G$ is the isometry group of $(M^n,g_{ij}(x) )$, then $G$ remains and isometric subgroup of $(M^n,g_{ij}(x,t) )$ for $t \in [0,T]$.
\end{theorem*}

Taken together, the previous two results show the following.

\begin{corollary} \label{ExistenceUniqueness}
Let $(M,g,D)$ be a complete Hessian manifold whose tangent bundle is trivial. Suppose also that the curvature of $\mathbb{M} = T M$ with the K\"ahler metric $\omega_0$ induced by $g$ and $D$ is bounded.  Then there exists a unique K\"ahler-Ricci flow on $\mathbb{M} = T M$ with initial metric $\omega_0$. Furthermore, this flow preserves the $\mathbb{R}^n$ symmetry and so is induced by a flow of the underlying convex potential $\Psi: \Omega \to \mathbb{R}$.
\end{corollary}

For K\"ahler-Sasaki metrics where the underlying Hessian manifold $M$ is compact (but whose tangent bundle need not be trivial), we can also establish short time existence and uniqueness of this flow. To do so, we pass to the affine universal cover $\Omega$ of $M$. Doing so, we make use of the following result of Shima.

\begin{theorem*}[\cite{Shima} Theorem B]
Let $M$ be a connected Hessian manifold. If there
exists a subgroup $G$ of automorphisms of $M$ such that $G \backslash M$ is quasicompact, then the universal covering manifold of $M$ is a convex domain
in a real affine space. 
\end{theorem*}

Since the affine universal cover is a convex domain, its tangent bundle is a tube domain, so we can make use of the previous existence and uniqueness result. Furthermore, the Chen-Zhu theorem implies that the deck transformations remain isometries, so we quotient them out to solve the K\"ahler-Ricci flow on the original Hessian manifold $M$. Alternatively, we can establish existence and uniqueness for compact Hessian manifolds by using the Hesse-Koszul flow, which we discuss in the next subsection.


These results establish short-time existence and uniqueness for K\"ahler-Ricci flow on the spaces we will study in this paper, so we turn our attention to understanding the behavior in the large-time limit.

Using the intuition that K\"ahler-Ricci flow behaves similarly to a reaction-diffusion equation for Riemannian metrics, we might hope to show that the normalized flow converges to a steady state (i.e., a K\"ahler-Einstein metric). For an arbitrary K\"ahler metric, this need not be the case (see work of Song and Tian \cite{SongTian} studying singularity formation). 
However, in a recent paper \cite{Tong}, Tong showed that this is the case for complete K\"ahler metrics which have negatively pinched holomorphic sectional curvature.

\begin{theorem*}[\cite{Tong} Theorem 1.1]
Suppose $(\mathbb{M}^{2n}, \omega_0)$ is a complete K\"ahler manifold whose holomorphic sectional curvature is negatively pinched. Then the K\"ahler-Ricci flow with initial condition $\omega_0$ exists for all time and the normalized flow converges to a K\"ahler-Einstein metric of scalar curvature $-2$ which is uniformly equivalent to the original metric.
\end{theorem*}

Using this result, Tong found an alternate proof for Wu and Yau's result on the existence of K\"ahler-Einstein metrics for negatively pinched spaces \cite{WuYau}. In this paper, we will use this theorem as a sufficient condition for the K\"ahler-Ricci flow to converge. This result does not imply that negative holomorphic sectional curvature is preserved under the flow. However, Wu and Yau showed that the holomorphic sectional curvature remains negatively pinched (with a weaker constant) for some definite amount of time.

It is worth emphasizing that Tong's result uses the assumption that the holomorphic sectional curvature is negatively pinched for all $(1,0)$ vectors. It seems likely that a version of this result holds for tube domains with negatively pinched polarized holomorphic sectional curvature given a sufficiently sharp lower bound on the orthogonal antibisectional curvature. Without the latter assumption, it does not seem possible to establish Royden's lemma \cite{Royden}, which plays an essential role in Tong's proof.

\subsection{The Hesse-Koszul Flow and affine geometry}

We can understand the K\"ahler-Ricci flow on tube domains in terms of a flow on the underlying Hessian manifold. Mirghafouri and Malek \cite{MirghafouriMalek} studied this flow, which evolves a Hessian manifold $(M, g_t ,D)$ by deforming the potential by the parabolic Monge-Amp\`ere equation\footnote{In their paper, the flow is expressed in terms of Koszul forms in affine geometry, but we will not use this terminology}:
\begin{equation} \label{HesseKoszulflow}
    \frac{\partial \Psi}{\partial t} = 2 \log \left( \det \left [ \frac{\partial \Psi}{\partial x^i \partial x^j}  \right ] \right ).
\end{equation}
Following the convention in \cite{PuechmorelTo}, we call this the \textit{Hesse-Koszul flow}. 
 Prior to this work, existence and uniqueness for the Hesse-Koszul flow had only been established for compact Hessian manifolds.\footnote{Corollary \ref{ExistenceUniqueness} provides short-time existence and uniqueness for the Hesse-Koszul flow for complete parallelizable Hessian manifolds whose curvature is bounded.} The long-time behavior of the Hesse-Koszul flow was recently studied by Puechmorel and T\^o \cite{PuechmorelTo}, who showed the following result.

\begin{theorem*}[\cite{PuechmorelTo} Theorem 2]
 Let $(M,g_0,D)$ be a compact Hessian manifold. Assume that the first affine Chern class of $M$ is negative. Starting from any Hessian metric g0, the normalized Hesse-Koszul flow exists for all time and converges in $C^\infty$ to a Hesse-Einstein metric\footnote{Hesse-Einstein metrics are the affine geometric equivalent of a K\"ahler-Einstein metrics.}
\end{theorem*}

By combining this result and Corollary \ref{NegativecostcurvedKahlerEinstein}, we obtain the following result, which is simply a restatement of Corollary \ref{CompactHessianEinstein} in the language of affine geometry.

\begin{corollary*}
Let $(M, g_0, D)$ be a compact Hessian manifold. Suppose that $g_0$ has non-positive anti-bisectional curvature and that the first affine Chern class of $M$ is negative. Then under the normalized Hesse-Koszul flow, $(M, g_t, D)$ converges to a Hesse-Einstein metric $g_\infty$ which is negatively cost-curved.
\end{corollary*}

For complete Hessian manifolds with bounded curvature\footnote{Bounded curvature here means that the tangent bundle has bounded curvature, not simply the underlying Hessian manifold.} but non-trivial tangent bundle, the existence and uniqueness of the Hesse-Koszul flow has not been established. Presumably, this can be done by adapting Shi's work to this setting, after which one can then apply the Chen-Zhu uniqueness theorem immediately. However, we will not consider these spaces in this paper.

\subsubsection{Advantages and disadvantages of the Hesse-Koszul flow}

For readers whose primary interest is affine geometry, it might be more natural to rephrase our results in terms of Hessian manifolds and the Hesse-Koszul flow. Indeed, the distinction between K\"ahler-Ricci flow for tube domains and the Hesse-Koszul flow is mainly a matter of terminology. Using affine geometry, it is simpler to define the notion of signed anti-bisectional curvature. Furthermore, some of the calculations simplify when stated in terms of Hesse-Koszul flow, as there is no need to pass back and forth from the tangent bundle to the underlying Hessian manifold.  However, there are two disadvantages of working directly in terms of the Hesse-Koszul flow.

First, the analytic results that we rely on were proven for Ricci flow. In order to work with Hesse-Koszul flow, it would necessary to translate these results into the language of affine differential geometry. In particular, rephrasing Tong's result in terms of affine geometry is quite awkward.

Second, one of our primary motivations for considering these flows are to understand the interaction between Ricci curvature and holomorphic sectional curvature. There is an extensive body of literature studying this question in K\"ahler geometry but no corresponding body of literature studying the relationship between the second Koszul form and the Hessian sectional curvature for Hessian manifolds\footnote{Shima defines the Hessian sectional curvature to be the negative of what we call the polarized holomorphic sectional curvature \cite{Shimabook}.}.
For these two reasons, we will state our results in terms of K\"ahler-Ricci flow and complex geometry.

\section{Negative anti-bisectional curvature is preserved by K\"ahler-Ricci Flow}
\label{Proofofmaintheorem}

In this section, we prove Theorem \ref{Main theorem}, which shows that K\"ahler-Ricci flow preserves negative anti-bisectional curvature. We want a version that applies to compact and non-compact spaces alike. Because of this, instead of simply checking the null-vector condition, we adapt an argument by Shi, which originally showed that non-negative bisectional curvature is preserved along K\"ahler-Ricci flow (Theorem 5.3 of \cite{Shi}). For readers who are interested in the case of compact Hessian manifolds, it is possible to skip Subsection \ref{Shimaximumprinciple} and simply use $\widetilde R$ in place of $A$ in Subsection \ref{VerificationofNVC}



\subsection{Uhlenbeck's trick}

Let $(\mathbb{M}, \omega_t)$ be a family of complete K\"ahler-Sasaki metrics which evolve according to the unnormalized K\"ahler-Ricci flow. As discussed in Subsection \ref{BackgroundonKRflow}, the curvature of $\omega_t$ evolves via the equation
\begin{eqnarray}
\frac{ \partial} {\partial t} R_{i \overline j k \overline \ell} & = & \Delta R_{i \overline j k \overline \ell} \nonumber \\
& & + 2 \left( R_{i \overline j p \overline r} R_{s \overline q k \overline \ell} h^{\overline p  q} h^{r \overline s} +  R_{i \overline \ell p \overline r} R_{s \overline q k \overline j} h^{\overline p  q} h^{r \overline s}-  R_{i \overline q k \overline r} R_{p \overline j s \overline \ell} h^{\overline p  q} h^{r \overline s} \right) \label{Essentialreaction} \\
& & - \left( R_{i \overline q} R_{p \overline j k \overline \ell} h^{\overline p  q} + R_{p \overline j} R_{i \overline q k \overline \ell} h^{\overline p  q}+ R_{k \overline q} R_{i \overline j p \overline \ell} h^{\overline p  q}+R_{p \overline \ell} R_{i \overline j k \overline q} h^{\overline p  q} \right) \label{Uhlenbeckterms2}
\end{eqnarray}

The reaction terms in this expression (lines \eqref{Essentialreaction} and \eqref{Uhlenbeckterms2}) are particularly unwieldy, so we make use of a technique known as Uhlenbeck's trick to simplify the expression. We consider an abstract vector bundle $V$ which is isomorphic to the complex tangent bundle $T_{\mathbb{C}} \mathbb{M}$. Furthermore, we fix a metric $\tilde h_{AB}$ on the fibers of $V$ which is $\mathbb{R}^n$-symmetric\footnote{With Uhlenbeck's trick, normally one can choose $\tilde h$ arbitrarily, Here, it is important to respect the translation symmetry.}. Using this symmetry, we also have a time invariant metric on $\Omega$ as well, which we denote $\tilde g$.

We choose an isometry $\mathcal{U} = \left \{ \mathcal{U}_i^j \right \} $ between $V$ and $T \mathbb{M}$ and let this isometry evolve by the equation
\begin{equation} \label{Uhlenbecktrick}
    \frac{\partial}{\partial t} \mathcal{U}_j^i = h^{ik} \textrm{Ric}_{k \ell} \mathcal{U}_j^\ell.
\end{equation}

Furthermore, we can pull over the Levi-Civita connection of the metric $h$ to the abstract vector bundle $V$ and compute the curvature tensor, which we denote $\widetilde{R}$. Although this is conceptually more complicated, the advantage of doing this is that it greatly simplifies the curvature evolution equations.

\[ 
\frac{\partial}{\partial t} \widetilde{R}_{i \overline j k \overline \ell}  =  \Delta \widetilde{R}_{i \overline j k \overline \ell}+ 2 \left( \widetilde{R}_{i \overline j p \overline r} \widetilde{R}_{s \overline q k \overline \ell} \widetilde{h}^{\overline p  q} \widetilde{h}^{r \overline s} +  \widetilde{R}_{i \overline \ell p \overline r} \widetilde{R}_{s \overline q k \overline j} \widetilde{h}^{\overline p  q} \widetilde{h}^{r \overline s}-  \widetilde{R}_{i \overline q k \overline r} \widetilde{R}_{p \overline j s \overline \ell} \widetilde{h}^{\overline p  q} \widetilde{h}^{r \overline s} \right)
\]

We now pick holomorphic coordinates $\{ z \}$ so that 
$\tilde h$ is the identity at a point $p \in \MM$. We also insist that the real parts of $z$ give affine coordinates of $\Omega$ (i.e., the coordinates respect the polarization of $\mathbb{M}$). In other words, we pick coordinates which respect the affine structure and translation symmetry of $\mathbb{M}$.
Using the fact that Ricci flow preserves the K\"ahlerness of the metric (Theorem 5.1 of Shi \cite{Shi}), the curvature evolution further simplifies (at $p$) to the following:
\begin{eqnarray}\label{Curvatureevolutionsimple}
\frac{ \partial} {\partial t} \widetilde R_{i \overline j k \overline \ell} & = & \Delta \widetilde R_{i \overline j k \overline \ell} + 2\widetilde R_{i \overline j p \overline q} \widetilde R_{q \overline p k \overline \ell} + 2\widetilde R_{i \overline \ell p \overline q} \widetilde R_{q \overline p k \overline j}-  2\widetilde R_{i \overline p k \overline q} \widetilde R_{p \overline j q \overline \ell}
\end{eqnarray}

We now define the quadratic $Q$ to be
\begin{equation} \label{DefinitionofQ}
    Q(\widetilde R)_{i \overline j k \overline \ell} = 2 \left ( \widetilde  R_{i \overline j p \overline q}  \widetilde R_{q \overline p k \overline \ell} +\widetilde R_{i \overline \ell p \overline q} \widetilde R_{q \overline p k \overline j}- \widetilde  R_{i \overline p k \overline q} \widetilde R_{p \overline j q \overline \ell} \right ),
\end{equation} 
so that the curvature evolution can be written as
\begin{equation}\label{CurvatureevolutionwithQ}
    \frac{ \partial} {\partial t} \widetilde R_{i \overline j k \overline \ell}  =  \Delta \widetilde R_{i \overline j k \overline \ell}  + Q(\widetilde R)_{i \overline j k \overline \ell}.
\end{equation}

By definition, the anti-bisectional curvature is non-positive if and only if
for any $v, w  \in T^{real} \Omega$,
\begin{equation}
    R(v^{pol.}, \overline w^{pol.}, v^{pol.}, \overline w^{pol.}) \leq 0.
\end{equation}

From the fact that our choice of coordinates on $V$ respect the tube domain structure and polarization, we have that the following: 
\begin{equation}\label{AnegativeifftildeAnegative}
    R(v^{pol.}, \overline w^{pol.}, v^{pol.}, \overline w^{pol.}) \leq 0 \iff \widetilde R(v^{pol.}, \overline w^{pol.}, v^{pol.}, \overline w^{pol.}) \leq 0. 
\end{equation}  

As such, it remains to show that the anti-bisectional curvature of $\widetilde{R}$ is non-positive. To do this, we make use of a maximum principle for complete spaces that was proven by Shi.

\subsection{Shi's Maximum Principle}
\label{Shimaximumprinciple}
For complete manifolds with bounded curvature that evolve under Ricci flow, Shi established the following maximum principle.

\begin{theorem*}[\cite{Shi} Theorem 4.8]
Suppose that $(\mathbb{M}^n, g(t))$ is a complete Ricci flow with uniformly bounded curvature at each time-slice. Suppose that $\varphi(z,t)$ is a $C^\infty$ function on $\mathbb{M} \times [0,T]$ such that

\begin{equation*}
  \left \{  \begin{aligned}
    \frac{\partial \varphi}{\partial t}=\Delta \varphi+C_{1}\left|\nabla_{k} \varphi\right|^{2}+Q(\varphi, z, t) & \text { on } \mathbb{M} \times[0, T] \\
    \varphi(z, t) \leq C_{2}<+\infty, \quad & \text { on } \mathbb{M} \times[0, T] \\
    \varphi(z, 0) \leq 0, \quad & \text { on } \mathbb{M} \\
    Q(\varphi, z, t) \leq C_{3} \varphi, \quad & \text { if } \varphi \geq 0,
    \end{aligned} \right.
\end{equation*}
where $C_1,\, C_2$ and $C_3$ are non-negative constants. Then we have
\begin{equation}
    \varphi(z, 0) \leq 0, \quad \text { on } \mathbb{M} \times[0, T]  
\end{equation}
\end{theorem*}

In order to apply this result, we must find an appropriate function $\varphi$ to apply this result. At each point in space-time $(z,t) \in \mathbb{M} \times [0,T)$, we consider the space of polarized unit vectors
 \begin{equation}
     S(z,t) :=  \{ \V = v^{pol.}~|~ v \in T^{real} \Omega \textrm{ with } \widetilde g_z(v,v)=1 \}, 
 \end{equation}
Note that the norm here is given with respect to the fixed metric $\widetilde g$, which is invariant in time. We then define a function
\begin{equation}
    \varphi(z,t) = - \sup \{ \theta \in \mathbb{R} ~|~ A( v^{pol.}, \overline{w^{pol.}} , v^{pol.}, \overline{w^{pol.}} ) (\theta, z,t) \leq 0 \textrm{ for all } v,w \in S(z,t) \} 
\end{equation}
where the tensor $A(\theta, z,t)$ is defined as
\begin{equation}\label{DefinitionofA}
    A(\theta)_{i \overline j k \overline \ell} = \widetilde R_{i \overline j k \overline \ell} + \theta \tilde g_{ik}  \tilde g_{\overline j \overline \ell}.
\end{equation}

Note that the definition of $A(\theta, z ,t)$ is slightly different from the definition of $A$ used by Shi. This is due to the difference between bisectional and anti-bisectional curvature. Also, note that we are passing freely between vectors on $\Omega$ and their lifts on $\MM$.

We then define the tensor $A(z,t)$ to be 
\begin{equation}
    A := A(-\varphi(z,t),z,t).
\end{equation} In other words, at each point of space-time which choose the supremum of $\theta$ so that $A(\theta,z,t)$ is non-positive definite at that point. The unusual sign convention for $\varphi$ is done so that we can use Shi's maximum principle verbatim without reversing inequalities.
To use the maximum principle, much of the proof follows Theorem 5.3 from \cite{Shi} with slight modifications. However, the proof of the null-vector condition (i.e., that $Q$ is ``inward pointing") is somewhat different, so we provide the details for that step.

\subsection{Verification of the null-vector condition}
\label{VerificationofNVC}

\begin{lemma}
Suppose $\{A_{i \overline j k \overline \ell} \}$ is a tensor which has the same symmetries as $\{R_{i \overline j k \overline \ell} \}$. As before, we set
\begin{equation}\label{Quadratic}
    Q(A)_{i \overline j k \overline \ell}  = 2\left(  A_{i \overline j p \overline q} A_{q \overline p k \overline \ell} + A_{i \overline \ell p \overline q} A_{q \overline p k \overline j}- A_{i \overline p k \overline q} A_{p \overline j q \overline \ell} \right),  
\end{equation}
where we assume $\tilde g$ is the identity at a point $z$ (and that the metric respects the translation symmetry). Suppose for a fixed point $z \in \mathbb{M}$, we have that
\begin{enumerate}
    \item $A(\X, \overline \Y, \X , \overline \Y) \leq 0$ for all $\X = x^{pol.},\Y = y^{pol.}$ with $x,y \in T^{real} \Omega$, and
     \item $A(\V, \overline \W, \V , \overline \W) = 0$ for some $\V = v^{pol.}, \W=w^{pol.}$.
\end{enumerate}
Then 
\begin{equation} \label{SignofQ}
Q(A)(\V, \overline \W, \V, \overline \W) \leq 0.
\end{equation}

\end{lemma}


\begin{proof}

To prove this result, we adapt an argument of Cao \cite{Cao} (but which was originally shown Mok \cite{Mok}). We start with two small lemmas.

\begin{lemma}\label{FirstderivA}
 For any vector $\X = x^{pol.}$,
\begin{equation}
    A_{\V \overline \X \V \overline \W} =A_{\X \overline \W \V \overline \W} = 0. 
\end{equation}
\end{lemma}

To see this, consider $G(s)= A(\V+s \X, \overline \W, \V+s\X, \overline \W)$. $G(s)$ is maximized at $0$, so $G^\prime(0) = 0$. Computing $G^\prime(0)$, we find the following:
\begin{equation}
    G^\prime(0) = R_{\X \overline \W \V \overline \W} + R_{\V \overline \W \X \overline \W} =  2 R_{\X \overline \W \V \overline \W}.
\end{equation}
Similarly, we can show the same for $\overline G(s) = R(\V, \overline \W+ s \overline \X, \V, \overline \W+s \overline \X) $.

\begin{lemma}
 For any two vectors $\X = x^{pol.}$ and $\Y = y^{pol.}$,
\begin{equation} A_{\Y \overline \W \Y \overline \W} +A_{\V \overline \X \V \overline \X} + 4 A_{\Y \overline \X \V \overline \W} \leq 0. \end{equation}
\end{lemma}

We consider the function 
\begin{equation}H(s) := A(\V+s \Y, \overline \W + s \overline \X, \V+s \Y, \overline \W+ s \overline \X).\end{equation}
Since $H$ is maximized at $s=0$, we have that $H(0)=H^\prime(0) =0$ and $H^{\prime \prime}(0)\leq 0$.
Calculating $H^{\prime \prime}$, we find that
\begin{eqnarray}
H^{\prime \prime}(0) & = & A_{\Y \overline \W \Y \overline \W} +A_{\V \overline \X \V \overline \X} \nonumber \\
& & + A_{\Y \overline \X \V \overline \W} +A_{\V \overline \W \Y \overline \X } + A_{\Y \overline \W \V \overline \X } + A_{\V \overline \X \Y \overline \W } \nonumber \\
& = & A_{\Y \overline \W \Y \overline \W} +A_{\V \overline \X \V \overline \X} + 4 A_{\Y \overline \X \V \overline \W}.
\end{eqnarray}

We can use this $H$ to define a semi-positive definite bilinear form $\mathcal{H}$. To do so, we consider an  an orthonormal basis $\{ e_i\}$ of $T^{real} \Omega$ and its lifts $\{\E_i = e_i^{pol.}\}$, which form a polarized unitary basis of $T \MM$. For vectors $\Xi = \xi^i \E_i$ and $\Z = \zeta^i \E_i$ (where all the coefficients are real), we define $\mathcal{H}(\Xi,\Z)$ as follows:
\begin{equation}\label{mathcalH}
    \mathcal{H}(\Xi,\Z) = A(\E_i, \overline \W, \E_j, \overline \W) \xi^i \xi^j +A(\V, \overline \E_i,\V, \overline \E_j) \overline \zeta^i \overline \zeta^j + 4 A(\E_i, \overline \E_j,\V, \overline \W) \xi^i \overline \zeta^j
 \end{equation}

The estimate on $H^{\prime \prime}$ shows that
\begin{equation}\label{mathcalHisnegative} \mathcal{H}(\Xi,\Z) \leq 0.
\end{equation} 

We now make use of the following lemma to show that $Q$ is non-positive. This lemma was originally stated by Cao (\cite{CaoHarnack} Lemma 4.1) and is attributed to Hamilton. However, there are some typos in the original proof\footnote{ In particular, the matrices $G_2$ and $G_3$ defined by Cao need not be semi-positive definite. For instance, if we set  \[G_1 = \left[ \begin{array}{cccc}
0 & 0 & 0  & 0 \\
0 & 1 & 1 &  0  \\
0 & 1 & 1 &  0  \\
0 & 0 & 0  & 0 \\
\end{array} \right], \] both $G_2$ and $G_3$ will have negative eigenvalues.}, so we provide a detailed argument.

\begin{lemma}\label{NullVectorCrux}
Let $M_1$ and $M_2$ be two $n \times n$ real symmetric semi-positive definite
matrices, and let $N$ be a real $n \times n$ matrix such that the $2n \times 2n$ real symmetric
matrix \begin{equation}
     G_1 =
\left[
\begin{array}{cc}
M_1  & N \\
 N^T & M_2  \\
\end{array}
\right]
\end{equation} 
is semi-positive definite. Then, we have that
\begin{equation} \label{positivetrace}
    \Tr(M_1 M_2 -N^2) \geq 0.
\end{equation}
\end{lemma}

\begin{proof}
Consider the matrix 

\begin{equation} G_2 =
\left[
\begin{array}{cc}
M_2  & -N^T \\
 -N & M_1  \\
\end{array}
\right].  \end{equation}

If we conjugate $G_1$ by the skew-symmetric orthogonal matrix
\begin{equation} J =
\left[
\begin{array}{cc}
0  & Id_n \\
 -Id_n & 0  \\
\end{array}
\right], \end{equation}
we find that
the matrix $G_2$ is similar to $G_1$;
\begin{equation} \label{conjugate matrices}
    G_2 = J G_1 J^T.
\end{equation}

As a result, $G_1$ and $G_2$ have the same characteristic polynomial. Furthermore, since both matrices are symmetric, $G_2$ is semi-positive definite whenever $G_1$ is.

As a result, whenever $G_1$ is semi-positive definite, we have that \begin{equation}\Tr(G_1 G_2) \geq 0. \end{equation} Calculating this explicitly, we find that 
\begin{equation} G_{1} G_{2}=\left(\begin{array}{cc}
M_1 M_2-N N & -M_1 N^T + N M_1 \\
N^T M_2 - M_2 N  & M_2 M_1-N^{T} N^T
\end{array}\right).\end{equation}
Taking the trace (and noting that $\Tr(N^2) = \Tr(N^T N^T)$, we obtain the desired result.
\end{proof}

To make use of Lemma \ref{NullVectorCrux}, we set 
\begin{enumerate}
    \item $(M_1)_{ij} = -A(\E_i,\overline \W, \E_j, \overline \W)$
    \item $(M_2)_{ij} = -A(\V, \overline \E_i,\V, \overline \E_j)$, and
    \item $(N)_{ij} = - 2 A(\V, \overline \W, \E_i,\overline \E_j)$.
\end{enumerate}

Since $M_1, M_2$ and   $-\mathcal{H}$ are semi-positive definite, Equation \eqref{positivetrace} implies that
\begin{equation} \label{Quadraticdominance}
    A(\E_i,\overline \W, \E_j, \overline \W) A(\V, \overline \E_i,\V, \overline \E_j) \geq  4  A(\V,\overline \W, \E_i, \overline \E_j ) A(\V,\overline \W, \E_j, \overline \E_i ).
\end{equation}

Furthermore, we have that
\begin{equation} \label{Traceofproduct}
    A(\E_i,\overline \W, \E_j, \overline \W) A(\V, \overline \E_i,\V, \overline \E_j) \geq 0,
\end{equation}
since this is the trace of the product of semi-positive definite matrices.
Combining \eqref{Quadratic}, \eqref{Quadraticdominance} and \eqref{Traceofproduct}, we see that
\begin{equation} \label{negativityofQ}
Q(A)(\V, \overline \W, \V, \overline \W) \leq 0. \end{equation}

Note that \eqref{Quadraticdominance} implies the the sharper estimate \begin{equation}\label{strongernegativityofQ} Q(A) (\V, \overline \W, \V, \overline \W)+  A(\E_i,\overline \W, \E_j, \overline \W) A(\V, \overline \E_i,\V, \overline \E_j)   \leq 0.\end{equation}
Although we did not use this stronger inequality here, it will play an important role in Corollary \ref{NegativecostcurvedKahlerEinstein}.
\end{proof}

From here, it is possible to follow the rest of Shi's argument for Theorem 5.3 verbatim without making any further changes. For brevity, we refer to that paper for the details.

\section{Geometric Consequences}
In this section, we provide several geometric consequences of Theorem \ref{Main theorem}.

\subsection{Positive time control in terms of the orthogonal anti-bisectional trace curvature}
We first show that if we start with a metric of negative anti-bisectional curvature, we have control of the holomorphic sectional curvature at positive times solely in terms of the orthogonal anti-bisectional trace curvature.

\begin{corollary*}
If $\mathbb{M}$ has negative anti-bisectional curvature and the scalar curvature at time $t=0$ satisfies $ S \geq -K$ for some $K>0$, then the holomorphic sectional curvature satisfies the estimate
    \begin{equation}
        0 \geq H(\X) \geq - C(n) \left( \frac{n}{t+\frac{n}{K}}+ 2 \mathfrak{O} \right).
    \end{equation}
\end{corollary*} 
  
  \subsubsection{Chen's Estimate}
To prove this result, we start by using a result of Chen (\cite{Chen} Corollary 2.3 ) to observe that along the Ricci flow, the scalar curvature $S$ satisfies the estimate
  \begin{equation}\label{scalarcurvatureestimate}
      S \geq - \frac{n}{t+\frac{n}{K}}.
  \end{equation} 
  
  Here, $K$ is a lower bound of the scalar curvature at the initial time. For compact manifolds, this inequality is a straightforward application of the maximum principle, but Chen proved it for complete manifolds as well. On its own, this estimate on its own does not give uniform control over the geometry, as one expects the scalar curvature to become very large and positive.
  However, since the polarized holomorphic sectional curvature remains negative, we can use the following lemma (which is a quantitative version of Berger's estimate on the scalar curvature) to prove Corollary \ref{uniformcontrol}.

\subsubsection{A Modified Berger's Lemma}

\begin{lemma}\label{QuantitativeBergers}
Suppose $\mathbb{M}$ is a Kahler manifold with non-positive polarized holomorphic sectional curvature. Then, there exists a constant $C(n)$ depending only on the dimension so that for all polarized unit $(1,0)$ tangent vectors $\X$, the scalar curvature satisfies the inequality
\begin{equation} \label{QuantitativeBergerinequality}
    S+2\mathfrak{O} \leq C(n) (H(\X)).
\end{equation}
\end{lemma}

Using the polarized unitary basis $\{ \E_i \}$ as before, the scalar curvature is given by
 \begin{equation}
      S = \sum R_{i\overline i j \overline j}
 \end{equation}
 and the orthogonal anti-bisectional trace curvature is given by
 \begin{equation} \mathfrak{O}= \sum_{i \neq j} R_{i \overline j i \overline j}  \end{equation}

We first prove a modified version of Berger's lemma, \cite{Berger} which shows that the scalar curvature is the average of the holomorphic sectional curvatures.
We let $d \theta(\Z)$ be the uniform probability measure on $\mathbb{S}^{n-1} \subset \left( T^{real}_z \Omega \right)^{pol.}$ and consider the integral
 \begin{equation}\label{averagepolarized}
      \int_{|\Z|=1, \Z \in \left( T_{z}^{real} \Omega \right)^{pol.}} R_{i\overline j k \overline \ell} \Z_i \Z_j \Z_k \Z_\ell \, d \theta(\Z). 
 \end{equation}

When we integrate polynomials over the sphere \cite{Folland}, we find that
\begin{equation}
\int_{S^{n-1}}|\Z_j|^4 \hspace{0.3mm}d \theta(\Z)=\frac{ 2 \Gamma[5/2] \, \Gamma[1/2]^{n-2}}{\Gamma[ \frac{3+n}{2}]},\quad j=1,\dots,n,
\end{equation}
and
\begin{equation}
\int_{S^{n-1}}|\Z_j|^2|\Z_k|^2\hspace{0.3mm}d \theta(\Z)=\frac{ 2 \Gamma[3/2]^2 \, \Gamma[1/2]^{n-3}}{\Gamma[ \frac{3+n}{2}]},\quad 1\le j\ne k\le n.
\end{equation}

As such, we have that
\begin{equation}
\int_{S^{n-1}}|\Z_j|^4 \hspace{0.3mm}d \theta(\Z)= 2 \int_{S^{n-1}}|\Z_j|^2|\Z_k|^2\hspace{0.3mm}d \theta(\Z).
\end{equation}

As a result, when we compute \eqref{averagepolarized}, all of the terms vanish except those where the indices appear in pairs (or are all the same). This implies the following:
\begin{eqnarray}
    \int_{|\Z|=1, \Z \in \left( T_{z}^{real} \Omega \right)^{pol.}} R_{i\overline j k \overline \ell} \Z_i \Z_j \Z_k \Z_\ell \, d \theta(\Z) & = & C(n) \sum_{i,j} R_{iijj}( 2\delta_{ij}+(1-\delta_{ij})) \nonumber \\
    & & + 2  C(n) \sum_{i \neq j} R_{i \overline j i \overline j} \nonumber \\
    & = & C(n) \left( S(z) + 2 \sum_{i \neq j} R_{i \overline j i \overline j} \right). \label{Integralformula}
\end{eqnarray}

This is where the terms on the left hand side of Inequality \eqref{QuantitativeBergerinequality} come from.

\subsubsection{Bounding the holomorphic sectional curvature from below}

To finish the proof of lemma \ref{QuantitativeBergers}, it suffices to show inequality \eqref{QuantitativeBergerinequality} for the polarized holomorphic tangent vector $\X$ which minimizes the holomorphic sectional curvature. As such, we must establish that the minimum holomorphic sectional curvature cannot be too much smaller than the average. 

Suppose $\X$ is a polarized vector which minimizes the holomorphic sectional curvature at $z$. Let $\X^\perp$ be a polarized unit holomorphic tangent vector perpendicular to $\X$. Then, we consider the polarized unit holomorphic tangent vector $\X_{\theta} = \cos \theta \X + \sin \theta \X^\perp$.

Define the function 
\begin{equation}\label{definitionoff}
    f(\theta) := R( \X_{\theta}, \overline{ \X_{\theta}}, \X_{\theta}, \overline{ \X_{\theta}}) .
\end{equation}

We write out $f(\theta)$ as a trigonometric polynomial. Since the curvature tensor is linear in each entry, the polynomial must be of the form
\begin{equation}
 f(\theta) = a_0 + a_1 \cos(2 \theta) +a_2 \sin(2 \theta)+ a_3 \cos(4 \theta) +a_4 \sin(4 \theta).   
\end{equation}

Since $f$ is minimized at $\theta = 0$, it follows that
\begin{equation}\label{firstderivf}
    a_2 + 2a_4 = 0,
\end{equation} 
\begin{equation}\label{secondderivf}
     a_1 + 4 a_3 \leq 0, \textrm{ and }
\end{equation}
  \begin{equation} \label{fcomparison}
     a_1 \leq 0. 
  \end{equation}
  The first two inequalities follow from computing the first and second derivatives of $f$. The third follows from the fact that $f(0)\leq f(\pi/2)$, and noting that the only term that changes in the trigonometric polynomial is $a_1$. By possibly switching $\X^\perp$ to $-\X^\perp$, we may assume $a_2 \geq 0$.

We now bound $a_0$ from above by considering $f(\pi/4)=a_0 + a_2 -a_3 $ and $f(\pi_2) = a_0 - a_1 +a_3 $. We consider two cases. 
\begin{enumerate}
    \item $a_1 \geq 3 a_3$, in which case $-a_1 \leq 3 a_3$ and so $a_2 -a_3 \geq -\frac{a_1+a_3}{4}$
    \item $a_1 \leq 3 a_3$, in which case $-a_1+a_3 \geq -\frac{a_1+a_3}{2}$.
\end{enumerate}

Since $f$ is non-positive, in either case we have that $a_0 \leq \frac{a_1 + a_3}{4}$. However, the average holomorphic sectional curvature is simply $a_0$, which shows that the average holomorphic sectional curvature in the $(\X,\X^\perp)$-plane is at most $1/5$ of the minimum holomorphic sectional curvature. 
Furthermore, since 
\begin{equation} a_0 + \frac{\sqrt{2}}{2}a_1 -\left(1-\frac{\sqrt{2}}{2} \right) a_2 = f(-\pi/8) \geq f(0) = a_0 +a_1+a_3, \end{equation}
it necessary follows that
\begin{equation}\label{a2bound} -\left(1-\frac{\sqrt{2}}{2} \right) a_1 -a_3 \geq  \left(1-\frac{\sqrt{2}}{2} \right) a_2.
\end{equation}

By combining \eqref{secondderivf} with \eqref{a2bound} and doing some straightforward numeric calculations, which find that for $|\theta| < \frac{\pi}{12}$, we have that

\begin{eqnarray}
f(\theta) &\leq& f\left(\frac{\pi}{12} \right)\nonumber \\
& = & a_0 + \frac{\sqrt{3}}{2}a_1 +\frac{a_3}{2} + \frac{2-\sqrt{3}}{4}a_2 \nonumber \\
& \leq & a_0 +\frac{\sqrt{3}}{2}a_1 +\frac{a_3}{2} + \frac{1}{14} a_2 \nonumber \\
& \leq & a_0.
\end{eqnarray}

As we established earlier, $a_0 \leq \frac{H(\X)}{5}$. As such, this shows that for any unit holomorphic tangent vector $\Y$ whose angle with $\X$ is less than $\frac{\pi}{12}$, the holomorphic sectional curvature satisfies
\begin{equation}\label{smallangleestimate}
    H(\X) \leq  \frac{ H(\X)}{5}.
\end{equation}

The set of holomorphic vectors whose angle with $\X$ is at most $\pi/12$ has positive measure in $\mathbb{S}^{n-1}$. Combining Equation \eqref{Integralformula} with Inequality \eqref{smallangleestimate}, we establish Corollary \ref{QuantitativeBergers}. \qed
\vspace{3mm}

 Before moving on, let us make several remarks about this result. First, for a metric with negative anti-bisectional curvature, it is possible to control the entire curvature tensor in terms of the polarized holomorphic sectional curvature and $\mathfrak{O}$. As such, this estimate shows that we can control the flow at positive times solely in terms of the function $\mathfrak{O}$.
 
 Second, Lemma \ref{QuantitativeBergers} can be adapted to arbitrary K\"ahler metrics with either non-negative or non-positive holomorphic sectional curvature. For example, we can adapt the proof to obtain the following result.

\begin{proposition*}
Suppose $\mathbb{M}$ is a Kahler manifold with non-negative holomorphic sectional curvature. Then, there exists a constant $C(n)$ depending only on the dimension so that for all unit $(1,0)$ tangent vectors $\X$, the scalar curvature satisfies the inequality
\[ S \geq C(n) H(\X). \]

Furthermore, for metrics with non-positive holomorphic sectional curvature, we obtain a similar result with the final inequality reversed.
\end{proposition*}

  \subsection{ K\"ahler-Einstein metrics with negative cost-curvature}
  
We can also use Theorem \ref{Main theorem} in combination with Tong's result on negatively curved spaces to find tube domains whose K\"ahler-Einstein metrics have negative anti-bisectional curvature. Corollary \ref{NegativeantibisectionalKahlerEinstein} (restated here for convenience) follows immediately from these two results.
  
  \begin{corollary*}
 Given a tube domain which admits a complete metric which of strongly negative holomorphic sectional curvature and anti-bisectional curvature, the unique complete K\"ahler-Einstein metric with scalar curvature -2 also has non-positive anti-bisectional curvature.
\end{corollary*}

When $\Omega$ is a smooth, bounded and strongly convex domain, it is possible to strengthen this result to show that the polarized holomorphic sectional curvature is negatively pinched.

\begin{corollary*}
Suppose that $\Omega$ is a smooth, bounded, and strongly convex domain and that $T \Omega$ admits a complete K\"ahler metric with negative anti-bisectional curvature and strongly negative holomorphic sectional curvature. Then the complete K\"ahler-Einstein metric also has negative cost-curvature.
\end{corollary*}


\begin{proof}

From Corollary \ref{NegativeantibisectionalKahlerEinstein}, the K\"ahler-Einstein metric has non-positive anti-bisectional curvature. To show that it has negative cost-curvature, we must show that the polarized holomorphic sectional curvature is strongly negative.

To do so, we note that for tube domains, the K\"ahler-Einstein metric $\omega_{K.E.}$ has negatively pinched holomorphic sectional curvature in a neighborhood of the strictly pseudoconvex boundary points (see \cite{vanCoevering} Proposition 3.5 for details \cite{vanCoevering}). When $\Omega$ is smooth and strongly convex, the entire boundary is strictly pseudoconvex, which shows that the K\"ahler-Einstein metric has negative cost curvature near its boundary.

As a result, when $\Omega$ is bounded, we can find a relatively compact open set $\Omega^\lambda \subset \Omega$ such that the holomorphic sectional curvature has a uniform lower bound on $T( \Omega \backslash \Omega^\lambda)$. For the sake of contradiction, suppose that $\omega_{K.E.}$ does not have negative cost-curvature. Since the holomorphic sectional curvature is strongly negative away from $T \Omega^\lambda$, $\Omega^\lambda$ is relatively compact and the metric is translation-invariant in the fibers, there must be a point $z \in T \Omega^\lambda$ and a polarized vector at $z$ (which we assume without loss of generality is $\E_1$) such that $H(\E_1) = 0$. 

We now consider the K\"ahler-Ricci flow with initial condition $\omega_{K.E.}$. Since $\omega_{K.E.}$ is K\"ahler-Einstein, we must have that \begin{equation}\frac{\partial}{\partial t} H(\E_1) \equiv 0. \end{equation}
On the other hand, since the anti-bisectional curvature of $\omega_{K.E.}$ is non-positive, by Equation \eqref{Curvatureevolutionsimple}, we have that
\begin{eqnarray}
\frac{ \partial} {\partial t} \widetilde H(\E_1) & = & \Delta \widetilde H(\E_1) + 4\widetilde R_{1 \overline 1 p \overline q} \widetilde R_{q \overline p 1 \overline 1} -  2\widetilde R_{1 \overline p 1 \overline q} \widetilde R_{p \overline 1 q \overline 1} \nonumber \\
& = & \Delta \widetilde H(\E_1) - \widetilde R_{1 \overline p 1 \overline q} \widetilde R_{p \overline 1 q \overline 1} \\ 
&  & + (4\widetilde R_{1 \overline 1 p \overline q} \widetilde R_{q \overline p 1 \overline 1} -  \widetilde R_{1 \overline p 1 \overline q} \widetilde R_{p \overline 1 q \overline 1}). \label{lastline}
\end{eqnarray}
From Inequality \eqref{Quadraticdominance}, \eqref{lastline} is non-positive. As such, we have that
\begin{eqnarray}
\frac{ \partial} {\partial t} \widetilde H(\E_1) & \leq & \Delta \widetilde H(\E_1) - \widetilde R_{1 \overline p 1 \overline q} \widetilde R_{p \overline 1 q \overline 1}.
\end{eqnarray}
Since $\Delta \widetilde H(\E_1) \leq 0$, we must have that $ \widetilde R_{1 \overline p 1 \overline q} \widetilde R_{p \overline 1 q \overline 1} \leq 0$ in order for $\frac{ \partial} {\partial t} \widetilde H(\E_1)$ to vanish.
Again appealing to Inequality \eqref{Quadraticdominance}, this implies that
\begin{eqnarray} \label{Vanishingquadratic}
   0 \geq \widetilde  R(\E_1,\overline \E_1, \E_i, \overline \E_j ) \widetilde R(\E_1,\overline \E_1, \E_j, \overline \E_i )
\end{eqnarray}

However, due to the fact that the metric is generated by a convex function on $\Omega$,
$\widetilde R(\E_1,\overline \E_1, \E_i, \overline \E_j )$ is symmetric in $i$ and $j$. As such, Inequality \eqref{Vanishingquadratic} shows that
\begin{equation} 0 \geq \sum_{i,j} |\widetilde R(\E_1,\overline \E_1, \E_i, \overline \E_j )|^2,\end{equation}
 which implies that for all $i$ and $j$,
\begin{equation}\label{vanishingbisectional}
  \widetilde  R(\E_1,\overline \E_1, \E_i, \overline \E_j) \equiv 0.
\end{equation}
As such, the holomorphic bisectional curvature $R( \E_1,\overline \E_1, \Z, \overline \Z)$ vanishes for all $(1,0)$-vectors $\Z$. We can then use the following modified version of Berger's lemma to calculation the Ricci curvature $\textrm{Ric}(\E_1,\overline \E_1)$.

\begin{lemma}\label{BergersRicciformula}
For any (1,0) vector $\X$, the Ricci curvature $Ric( \X,\overline \X)$ is given by the integral
\begin{equation}\label{Ricciasaverage}
   \operatorname{Ric}(\X, \overline \X) = \frac{n-1}{Vol( \mathbb{S}^{2n-3})} \int_{|\Z|=1, \Z \in T_{z}^{(1,0)} \mathbb{M}} R(\X, \overline \X,\Z, \overline \Z )\, d \theta(\Z).
\end{equation}
\end{lemma}

 Combining \eqref{vanishingbisectional} and \eqref{Ricciasaverage}, it follows that $\textrm{Ric}(\E_1,\overline \E_1) =0$, which contradicts the assumption that $\omega_{K.E.}$ is K\"ahler-Einstein with negative scalar curvature.
\end{proof}

As we noted in the introduction, it is possible to prove Corollary \ref{NegativecostcurvedKahlerEinstein} when the base is a compact Hessian manifold, in which case we can use the compactness of $M$ to use the standard tensor maximum principle.
It is of interest to determine whether this result can be sharpened to obtain quantitative estimates on the pinching of the polarized holomorphic sectional curvature. However, at present we are not able to do so.

\subsection{Towards a differential Harnack inequality}
\label{AntibisectionalandRicci}
Although negativity of the anti-bisectional curvature is a strong assumption, its relation to other curvature conditions is not straight-forward. For instance, it does not imply negative Ricci curvature, even in the special case where the potential is $O(n)$-symmetric (which were studied previously by the authors \cite{KhanZhangZheng}).
The lack of straightforward relationship between the anti-bisectional curvature and the Ricci curvature has an important consequence which pertains to differential Harnack estimates.

Differential Harnack estimates are some of the most important results for Ricci flow. The original version of this inequality was proved by Hamilton and gives $C^1$ control over the scalar curvature for positive time whenever the initial metric has positive curvature operator \cite{HamiltonLiYauHarnack}. Cao proved another version for K\"ahler-Ricci flow under the assumption that the bisectional curvature is positive \cite{CaoHarnack}. For metrics of mixed curvature, Perelman found a related inequality \cite{Perelmanentropy} which involves coupling the Ricci flow with a conjugate heat equation.

In order to find such inequalities, the general strategy is to consider a gradient Ricci soliton and find tensor expressions which vanish. Doing so provides candidate tensors which may have a sign along the flow.
The soliton calculations done by Cao and Hamilton use the fact the Ricci curvature has a sign to ensure that their expression is positive for small times, so cannot be used for the anti-bisectional curvature.
It would be of interest to find a similar inequality for metrics with negative anti-bisectional curvature, especially since doing so seems to require a different approach.

\begin{question}
Can we establish a differential Harnack inequality for K\"ahler-Ricci flow for metrics of negative anti-bisectional curvature?
\end{question}

\section{Orthogonal anti-bisectional curvature and K\"ahler Ricci Flow}
\label{NOABSection}
We now study the behavior of orthogonal anti-bisectional curvature under K\"ahler-Ricci flow. Recall that in optimal transport, the cost $\Psi(x-y)$ is weakly regular whenever the orthogonal anti-bisectional curvature is non-negative.

\begin{theorem*}
Suppose that $ \Omega \subset \mathbb{R}^2 $ is a convex domain and $\Psi :\Omega \to \mathbb{R}$ is a strongly convex function so that the K\"ahler manifold $(T \Omega, \omega_\Psi)$ 
\begin{enumerate}
    \item is complete,
    \item has bounded curvature, and
    \item has non-negative orthogonal anti-bisectional curvature.
\end{enumerate}
 Then the orthogonal anti-bisectional curvature remains non-negative along K\"ahler-Ricci flow. 
\end{theorem*}

\begin{proof}

The proof is nearly identical to the proof that non-positive anti-bisectional curvature is preserved, except for the verification of the null-vector condition. As such, we will only provide the details for that step.

As before, we use Uhlenbeck's trick so that the curvature evolves by the equation. 
\begin{eqnarray}
\frac{ \partial} {\partial t} \widetilde R_{i \overline j k \overline \ell} & = & \Delta \widetilde R_{i \overline j k \overline \ell} + 2\widetilde R_{i \overline j p \overline q} \widetilde R_{q \overline p k \overline \ell} + 2\widetilde R_{i \overline \ell p \overline q} \widetilde R_{q \overline p k \overline j}-2 \widetilde R_{i \overline p k \overline q} \widetilde R_{p \overline j q \overline \ell}
\end{eqnarray}

We then define a function
\begin{eqnarray*}
    \varphi(z,t) & \\ =  &\inf \left \{ \theta \in \mathbb{R} ~|~ A_{ \V \overline \W \V \overline \W} (\theta, z,t) \geq 0 \textrm{ for } \V = v^{pol.},\W = w^{pol.}  \textrm{ with } \tilde g(v,w)=0 \right \} 
\end{eqnarray*}
where the tensor $A(\theta, z,t)$ is defined as
\begin{equation} A(\theta)_{i \overline j k \overline \ell} = \widetilde R_{i \overline j k \overline \ell} +  \theta \widetilde g_{ik}  \widetilde g_{\overline j \overline \ell}. \end{equation}

As before, we define the tensor 
\begin{equation}A(z,t) := A(\varphi(z,t),z,t), \end{equation}
which is now positive definite when restricted to lifts of pairs of orthogonal vectors. 

\begin{lemma} Suppose that $\{A_{i \overline j k \overline \ell}\}$ is a tensor with the same symmetries as $\{ \widetilde R_{i \overline j k \overline \ell}\}$.
Suppose that $\tilde g$ is the identity at one point and that $\V = v^{pol.},\W = w^{pol.}$ are vectors with $\tilde g(v,w) = 0$.

Suppose that
\begin{equation}
    A(\V, \overline \W, \V, \overline \W) =0,
\end{equation}  
and that for all other orthogonal pairs of vectors $\xi$ and $\eta$, \begin{equation}
     A(\xi^{pol.}, \overline \eta^{pol.}, \xi^{pol.}, \overline \eta^{pol.}) \geq 0.
\end{equation}

Then, if we let $Q$ be as in \eqref{DefinitionofQ}
\begin{equation} Q(A)(\V, \overline \W, \V, \overline \W) \geq 0.\end{equation}
\end{lemma}

By an affine change of coordinates, we can assume $\V= e_1^{pol.}$ and $\W=e_2^{pol.}$. Note that $e_1 +s e_2 \perp e_2- \overline s e_1$. As such, we have that
\begin{equation}\frac{d}{ds} A( (e_1+s e_2)^{pol.}, \overline{(e_2 -s e_1)^{pol.}}, ( e_1+s e_2)^{pol.}, \overline{(e_2 -s e_1)^{pol.}}) \vert_{s=0} = 0 \end{equation}
Simplifying the expression, we find the following.
\begin{equation} \label{firstderivativeNOAB}
    A_{1 \overline 2 1 \overline 1}= A_{1 \overline 2 2 \overline 2}
\end{equation}

Computing $Q(A)$, we find that \begin{eqnarray}
Q(A)_{1 \overline 2 1 \overline 2} &= & 2   A_{1 \overline 2 p \overline q}   A_{q \overline p 1 \overline 2} +  2 A_{1 \overline 2 p \overline q}   A_{q \overline p 1 \overline 2}-  2 A_{1 \overline p 1 \overline q}   A_{p \overline 2 q \overline 2} \nonumber \\
& = & 4   A_{1 \overline 2 p \overline q}   A_{q \overline p 1 \overline 2} - 2  A_{1 \overline p 1 \overline q}   A_{p \overline 2 q \overline 2} \nonumber \\
& = & 4   A_{1 \overline 2 1 \overline 1}   A_{1 \overline 1 1 \overline 2} + 4   A_{1 \overline 2 2 \overline 2}   A_{2 \overline 2 1 \overline 2} \textrm{ when } p = q\textrm{ (the other terms vanish)} \nonumber  \\
& & -  2 A_{1 \overline 1 1 \overline 2}   A_{1 \overline 2 2 \overline 2} -  2 A_{1 \overline 2 1 \overline 1}   A_{2 \overline 2 1 \overline 2}  \textrm{ when } p \neq q\textrm{ (the other terms vanish)} \nonumber \\
& = & 4  (  A_{1 \overline 1 1 \overline 2})^2. \label{Qidentity}
\end{eqnarray}

Identity \eqref{Qidentity} uses Equation \eqref{firstderivativeNOAB} to simplify the expression. This is non-negative, which completes the proof.

\end{proof}

 For $n \geq 3$, the quadratic $Q$ has terms which are negative definite, and so $Q$ need not have a sign. As such, this result is special to complex dimension two.

\subsection{Applications}
\label{NOABApplications}

From a geometric standpoint, the assumption of non-negative orthogonal anti-bisectional curvature alone does not provide strong control over the geometry of a tube domain. For instance, it is possible to find metrics with non-negative orthogonal anti-bisectional curvature which do not converge to a K\"ahler-Einstein metric under the flow. As such, the natural field of applications for Theorem \ref{SurfaceswithNOAB} is optimal transport rather than complex geometry. 

\subsubsection{Monge's cost $c(x,y) = |x-y|$}
Since K\"ahler-Ricci flow tends to smoothen metrics, one natural application of the flow is to prove regularity estimates for for optimal transport when the cost function is rough. In order to motivate this question and give a concrete example of such a function, consider the cost 
 \begin{equation} \label{Mongecost}
     c_0(x,y) = \|x-y\|,
 \end{equation}
which was used in Monge's original paper on optimal transport \cite{Monge}.

From the perspective of the modern regularity theorem, this function is quite pathological. In particular, it fails to be $C^1$ at the origin and is degenerate elsewhere (in that the kernel of $ c_{i,j}(x,y)$ is non-trivial).
These properties affects the behavior of optimal transport. For instance, the optimal transport with respect to the cost $c_0$ generally fails to be unique. Even more strongly, when we consider the Kantorovich relaxation, it is possible to find optimal couplings which are non-deterministic, in that they are not supported on the graph of a function $T:X \to Y$.

Nevertheless, Monge's cost can be uniformly approximated (in $C^\alpha$) by the costs 
\begin{equation} \label{Mongedeformation}
     c_\epsilon = \sqrt{ \epsilon + \|x-y\|^2},
\end{equation}
which are smooth, non-degenerate, and have non-negative MTW tensor. As a result,
it is possible to develop a type of regularity theory for Monge's cost in terms of the ray-monotone optimal transport map (see Chapter 3 of Santambrogio \cite{Santambrogio} for details). As such, Monge's cost is much better behaved than a generic Lipschitz cost function. 
However, the deformation from $c_0$ to $c_\epsilon$ is somewhat ad hoc from a conceptual standpoint, and it would be preferable to have a canonical way to deform the cost function. For this purpose, K\"ahler-Ricci flow seems to be a natural candidate. 

At present, it is not possible to do this, since Monge's cost corresponds to a ``K\"ahler metric" which is singular at the origin and degenerate elsewhere. Furthermore, this ``metric" is not complete. However, if we can find a K\"ahler-Ricci flow whose initial condition corresponds to the Monge's cost in some weak sense, we would be able to fit Monge's cost into the existing regularity theory. 

\subsubsection{Weak regularity using the K\"ahler-Ricci flow }

At first, it seems that in order to satisfy the MTW(0) condition, a cost must be $C^4$ (in that $\frac{\partial^2}{\partial x^i \partial x^j} \frac{\partial^2}{\partial y^k \partial y^l} c(x,y)$ is well defined). However, there are results to suggest that it is possible to obtain regularity for less smooth cost functions. For instance, Villani showed that the MTW(0) condition for the squared-distance cost is stable under Gromov-Hausdorff convergence \cite{VillaniStability}. Furthermore, Guillen and Kitagawa \cite{GuillenKitagawa} found a synthetic version of the $MTW(0)$ condition known as quantitative quasi-convexity which only requires that the potential be $C^3$. For complete $\Psi$-costs in two dimensions, KR weak regularity (Definition \ref{KRWeakregularity}) provides another version of $MTW(0)$, which is meaningful for even less regular costs. 

 Theorem \ref{SurfaceswithNOAB} shows that for cost functions which are sufficiently smooth, this definition is equivalent to the normal MTW(0) condition. However, it has the advantage of being meaningful whenever the parabolic flow \eqref{Psicostflow} exists and smoothens the convex potential.



It is natural to ask whether KR weak regularity provides a meaningful generalization of weak regularity. In other words, these two definitions are equivalent for costs which are sufficiently smooth. However, unless we can say something about the optimal transport when the cost function is less smooth than $C^4$, there is no reason to use the K\"ahler-Ricci flow to define weak regularity. To address this question, we provide the following result.
 
 \begin{corollary} \label{Weak continuity}
 Let $(X,\mu)$ and $(Y,\nu)$ be probability spaces and $\Psi_0: \Omega \to \mathbb{R}$ be a strongly convex $W^{2,p}$ function (with $p>2$) such that following assumptions hold.
 \begin{enumerate}
 \item $X,~Y$ and $\Omega$ are subsets of $\mathbb{R}^2$.
    \item $X-Y \subset \Omega$.
     \item $\Psi$ induces a complete Hessian metric on $\Omega$.
     \item $X$ and $Y$ are smooth and bounded.
     \item There exists $\epsilon >0$ such that whenever $t <\epsilon$, $X$ and $Y$ are uniformly strongly relatively $c$-convex with respect to the cost $c(x,y) =\Psi_t(x-y)$ where $\Psi_t$ solves the flow \eqref{Psicostflow}. For a definition of relative $c$-convexity, see \cite{MTW} Definition 2.3.
     \item $\Psi_0$ is KR weakly regular (in the sense of Definition \ref{KRWeakregularity}).
     \item The measures $\mu$ and $\nu$ are absolutely continuous with respect to the Lebesgue measure and $\log \left( \frac{d\mu(x)}{d \nu(y)} \right) \in L^\infty( X \times Y) $.
 \end{enumerate}
 
We then fix an open set $X^\lambda$ which is relatively compact in the interior of $X$. Then there is an optimal transport $T_0$ with respect to the cost $c_0(x,y) = \Psi_0(x-y)$. Furthermore, there exists constants $\alpha^\prime \in (0,1)$ and  $C>0$ so that
        \begin{equation} \label{Holderestimate}
            \|T_0\|_{C^{\alpha^\prime}(X^\lambda)} < C.
        \end{equation}  
    
 The constant $\alpha^\prime \in (0,1)$ depends on $p$, the subset $X^\lambda \subset X$ and the $L^\infty$-norm of $\log \left( \frac{d\mu}{d \nu} \right)$. The H\"older norm depends the $W^{2,p}$ norm of $\Psi_0$ on $X-Y$.
 \end{corollary}

\begin{proof}

Consider the transport $T_t$ which is optimal with respect to the cost $c_t(x,y)=\Psi_t(x-y)$ for $t>0$. When $\Psi_0$ is $C^{2,\alpha}$, $\Psi_t$ will be $C^\infty$ for $t>0$ and bounded in $C^{2,\alpha}$ for small time. By Theorem \ref{SurfaceswithNOAB}, $c_t$ will also satisfy the MTW(0) condition.

Using the fifth assumption, it is possible to obtain a H\"older estimate for the transport $T_t$ on $X^\lambda$ for small times $t$. To do so, we first note that the $W^{2,p}$ estimate on the cost implies a $C^{1,\alpha}$-estimate via Morrey's inequality for some $\alpha >0$. 
Using this, we appeal to Theorem 9.5 of Figalli, Kim, McCann \cite{FigalliKimMcCann} to obtain a uniform H\"older estimate on $T_t$ for $t>0$. It is worth noting that the original proof of these result uses a $C^2$-estimate on $c$. However, the $C^2$-estimate only appears as an error term, so it is straightforward to adapt the proof to only require a $C^{1,\alpha}$-estimate, at the possible expense of weakening the final H\"older exponent\footnote{To apply the Figalli-Kim-McCann argument, we can replace the error term $\left\|D_{x x}^{2} c\right\|_{L^{\infty}(U \times V)} s^{2}|v|^{2}$ by $\left\| D_x c\right\|_{C^\alpha}(X \times Y) s^{1+\alpha}|v|^{1+\alpha}$. To apply the rest of the argument, we must have that $\alpha$ satisfies $\alpha > 1/K$, where $K$ is the constant determined by the dimension $n$ and the $L^\infty$ bound on  $\log \left( \frac{d\mu(x)}{d \nu(y)} \right)$. When $\alpha \leq 1/K$, we can instead set $K^\prime = 2/\alpha$ and repeat the rest of the argument in terms of $K^\prime$. This will decrease the final H\"older exponent, but otherwise leaves the argument unchanged.}
We can find a sequence $t_k \to 0$ so that the transport maps $T_{t_k}$ converge on $X^\lambda$ in $C^{ \alpha^\prime}$ for $ \alpha^\prime<\alpha$ to a limiting function $T_0$.

By the Kantorovich theorem, we know that an optimal $c_0$-coupling between $(X,\mu)$ and $(Y,\nu)$ exists (i.e., a coupling satisfying \eqref{Kantorovichproblem}). This coupling is unique and deterministic, by Theorem 10.28 of Villani \cite{OTON}. As a coupling, it depends continuously on the cost function in the weak-${}^\ast$ topology (see Corollary 5.20 for a precise statement \cite{OTON}). As such, we can pick a sub-sequence so that the transport maps $T_{t_k}$ converge to an optimal coupling $\pi_0$ with respect to $c_0$. However, the support of a $\pi_0$ is contained within the Kuratowski limit inferior of the supports of $T_{t_k}$, so in particular is contained within $T_0$. As such, $T_0$ is the optimal transport with respect to $c_0$ restricted to $X^\lambda$, and is bounded in $C^{ \alpha^\prime}$.
\end{proof}

Corollary \ref{Weak continuity} gives a continuity result for optimal transport when the cost function is $W^{2,p}$ and the flow \eqref{Psicostflow} exists. 
It seems likely that this flow is well defined for convex functions $\Psi_0$ which are $W^{2,p}$, which would allow us to remove the second assumption. For the purpose of comparison, it is worth discussing the corresponding situation for compact K\"ahler manifolds.
We cannot directly apply these results to our situation, but this does give some indication of what might be in our setting. Song and Tian \cite{SongTian} showed that if $\mathbb{M}$ is a compact K\"ahler manifold with K\"ahler form $\omega_0$ and volume form $\Omega$, the K\"ahler-Ricci flow exists, is unique and immediately smoothens the metric whenever the initial data is in the space $PSH(\mathbb{M}, \omega_0, \Omega)_p$ for some $p>1$. Here, $PSH(\mathbb{M}, \omega_0)$ is the space of all quasi-plurisubharmonic functions, which are upper semi-continuous functions on $\mathbb{M}$ satisfying $\omega_0+\sqrt{-1} \partial \bar{\partial} \psi \geq 0$. The space  $PSH(\mathbb{M}, \omega_0, \Omega)_p$ is defined as
 \begin{eqnarray} PSH(\mathbb{M}, \omega_0, \Omega)_p  \nonumber &\\
 = & \left \{  \psi \in PSH(\mathbb{M}, \omega_0) \cap L^\infty(\mathbb{M}) ~|~ \frac{\left(\omega_{0}+\sqrt{-1} \partial \bar{\partial} \psi \right)^{n}}{\Omega} \in L^{p}(\mathbb{M}) \right \}. 
 \end{eqnarray}

In our context, this is roughly equivalent to the Monge-Amp\'ere measure associated with $\Psi_0$ having a Radon-Nikodym derivative with respect to the Lebesgue measure which is $L^p$ for some $p>1$. However, showing that the flow exists and smoothens the metric in the non-compact case requires careful analysis, since it is not possible to use the maximum principle.

\subsubsection{General cost functions}

There are several limitations to Definition \ref{KRWeakregularity}. For instance, we are restricted to considering costs whose associated K\"ahler metric is complete. This restriction does not seem to be fundamental and can be solved by finding boundary conditions for K\"ahler-Ricci flow so that Theorem \ref{SurfaceswithNOAB} still holds.

The other major restriction is that we can only consider costs which are induced by a convex function. At present, this seems to be a fundamental obstacle; the associated geometry for a general cost function is pseudo-Riemannian, not K\"ahlerian \cite{KimMcCann, KimMcCannWarren}. For pseudo-Riemannian metrics, the Ricci flow is no longer weakly parabolic. As a result, we do not have a canonical way to evolve general cost functions. 



\section{Examples} \label{examples}

Initially, one might be concerned that non-positive anti-bisectional curvature is a very strong assumption, and that Theorem \ref{Main theorem} has limited application. In this section, we provide various examples of metrics satisfying this assumption, which suggests that such metrics exist in abundance.

\subsection{Simple examples}

We start by providing two simple examples.

\subsubsection{An example in one dimension}
\label{Simpleexamples}
Consider the one-dimensional domain $ \Omega =\mathbb{R}_{> 0}$ with the potential \begin{equation}
    \Psi(x)= -  \log(x).
\end{equation}
    
The tube domain $T \mathbb{R}_{> 0}$ is the half space $\mathbb{H}$ and the lift $\Psi^h$ induces its standard hyperbolic metric. This is a metric of constant negative holomorphic sectional curvature, so the anti-bisectional curvature is negative (as a Riemann surface, there is no orthogonal anti-bisectional curvature).
    
  We can also use this example to obtain a \textit{compact} Hessian manifold. Namely, if we consider the $\mathbb{Z}$-action $\psi$ on $\mathbb{R}_{> 0}$ given by $\psi(k,  x)=  2^k x$ for $k \in \mathbb{Z}, x \in \mathbb{R}_{> 0}$, we can consider the quotient manifold, which gives an affine structure on the circle which is distinct from the standard one. 
  

    \subsubsection{Two versions of the bidisk}
    Using the previous example, we can find a Hessian manifold whose tube domain is biholomorphically isometric to the bidisk (with its product metric) by considering the domain \begin{equation}
        \Omega = \{ x \in \mathbb{R}^2 ~|~ x_1,x_2 >0 \}
    \end{equation} with Hessian metric \begin{equation}\Psi(x) = - \log(x_1) - \log(x_2)  \end{equation}
    
    However, this potential is not unique. For instance, the potential
   \begin{equation} \Psi = -\log( \cos(x_1) + \cos(x_2)), \end{equation}
    defined on the domain   \begin{equation}\Omega = \{ |x_1| + |x_2| < \pi/2 \} \end{equation}
  also lifts to a metric which is biholomorphically isometric to the bidisk.
  
    This metric has negative holomorphic sectional curvature and non-positive anti-bisectional curvature (so is negatively cost-curved).

  \subsection{Other examples}
    
 We will now provide examples which require more computation, and do not just follow from the basic properties of the anti-bisectional curvature. Since computing the anti-bisectional curvature is fairly involved, we have written a Mathematica notebook which can be used to do the calculation and numerically check for negativity \cite{MTWNotebook}.
    
    \subsubsection{Calabi's example}
   In \cite{Calabi}, Calabi studied the tube domain $T B$ where $B$ is a unit ball in $\mathbb{R}^n$. In particular, he showed that this space admits a complete K\"ahler-Einstein metric which is not homogeneous. He also found a semi-explicit expression for this metric in terms of the solution to an ordinary differential equation.
   
   Using Corollary \ref{NegativecostcurvedKahlerEinstein}, we can show that this metric has negative cost-curvature.
  In order to do so, it is sufficient to find a single metric on this space which has non-positive anti-bisectional curvature and negative holomorphic sectional curvature. For this purpose, we consider the potential
       \begin{equation} \Psi(x) = -\log \left( 1-\sum_{i=1}^n (x_i)^2 \right). \end{equation}

The metric $\omega_\Psi$ becomes asymptotically hyperbolic near the boundary, and there are orthogonal vectors whose anti-bisectional curvature vanishes to third-order as one approaches $\partial T \mathbb{B}$. However, the metric $\omega_\Psi$ is negatively cost-curved. Furthermore, it also has negative holomorphic sectional curvature.  We postpone a derivation of this example to the appendix. 

    \subsubsection{A cone in $\mathbb{R}^2$}
    There are tube domains whose base is non-compact which admit negatively cost-curved metrics. For example, on the cone \begin{equation}
        C  =\{ (x_1,x_2) ~|~ |x_1| > |x_2|\},
    \end{equation} the convex potential
   \begin{equation} \Psi= -\log(x_1^2-x_2^2)  \end{equation}
lifts to a complete metric whose anti-bisectional curvature satisfies the identity
   \begin{equation} 2 \mathfrak{A}(v^{pol.}, (w^\sharp)^{pol.}) = -1 + \sin[2 \theta] \sin[2 \phi], \end{equation} 
where
\begin{eqnarray} v &=& \cos[\theta] \frac{\partial}{ \partial {x_1}} +\sin[\theta] \frac{\partial}{ \partial {x_2}}, \\
 w &=& \sin[\phi] dx_1 -\cos[\phi] dx_2.\end{eqnarray} 

Hessian metrics on convex cones have been previously studied (see, e.g., Chapter 4 of \cite{Shimabook} and \cite{Vinberg}). It would be of interest to find other cones whose tangent bundles admit negatively cost-curved metrics.

\subsection{Surfaces with non-negative orthogonal anti-bisectional curvature}

In recent work of the authors and J. Zhang \cite{KhanZhangZheng}, we studied the problem of finding $O(n)$-symmetric metrics on tube domains whose orthogonal anti-bisectional curvature is non-negative. The motivation for this was to find weakly regular cost functions which only depended on the Euclidean distance between points. In that paper, we showed that for such a metric to be complete, the underlying Hessian manifold must be $\mathbb{R}^n$. We also found several examples, such as the lift of the convex potential
\begin{equation} \Psi(x)= \|x \| - C \log(\|x\|+C ) \textrm{ for } C>0. \end{equation}

\subsection{Solitons and Statistical Mirror Symmetry}
\label{SolitonExample}
In recent work of Zhang and the first named author \cite{ZhangKhan}, we study T-duality in the context of tube domains. We refer to this phenomena as ``statistical mirror symmetry," due to its roots in information geometry.

As a particular example, we show that the Siegel-Jacobi space $\mathbb{H} \times \mathbb{C}$ with its special-affine invariant metric is mirror to the Siegel half-space, which is a space of constant negative holomorphic sectional curvature on the tube domain $T \Omega$ where
\begin{equation} \Omega = \{ (x_1,x_2) ~|~ x_1 > x_2^2\} . \end{equation}

In forthcoming work, we solve the K\"ahler-Ricci flow explicitly on both spaces and show that they remain coupled eternally. 
Furthermore, the Siegel-Jacobi space is a complex two-dimensional K\"ahler-Ricci soliton with non-negative orthogonal bisectional curvature and the Siegel half-space is a negatively cost-curved K\"ahler-Einstein metric (see \cite{ZhangKhan}, Subsubsection 4.4.4). As such, these two flows provide examples which satisfy the curvature conditions studied in this paper.



 \section{Acknowledgements}
 
 The first named author would like to thank Robert McCann, Mizan Khan, and Jun Zhang for some helpful suggestions. He was partially supported by AFOSR Grant FA9550-19-1-0213 (``Brain-Inspired Networks for Multifunctional Intelligent Systems and Aerial Vehicles'', MURI subcontract from UCLA). The second named author is partially supported by a research grant from NSFC with grant number 12071050.

\begin{appendix}
\section{Verification of Calabi's example}

In this section, we verify that the potential $\Psi= \log \left( 1 - \sum_{i=1}^n x_i^2 \right)$ lifts to a negatively cost-curved metric. We start by considering the case in complex dimension two. Let $\Omega$ be the unit ball in $\mathbb{R}^2$ and consider the point $(x^1,x^2) \in \Omega$. Let $(v,w)$ be the vector-covector pair \[ v = \cos(\theta) \frac{\partial}{\partial x_1} + \sin(\theta)  \frac{\partial}{\partial x_2} \] 
 \[ w = \sin(\phi) dx_1 - \cos(\phi) dx_2. \]
 
 Since the potential is rotationally symmetric, we can rotate the space so that $x_2=0$. Calculating the anti-bisectional curvature, we find the following.
 
\begin{eqnarray*} \mathfrak{A}\left( v^{pol.}, \left(w^\sharp \right)^{pol.} \right)  &= &  \cos[2 (\theta - \phi)] -\frac{2 (1 + x_1^2 + 2 x_1^4)}{(1 + x_1^2)^3} +  \frac{
 2 (-x_1^2 + x_1^4) \cos[2 \phi]]}{(1 + x_1^2)^3}  \\
\end{eqnarray*}

We now want to show that this expression is negative. To do so, note that it will be maximized in $\theta$ whenever $\phi = \theta$ (i.e., $g(v,w^\sharp)=0$). As such, we have that
\begin{eqnarray*} \mathfrak{A}\left( v^{pol.}, \left(w^\sharp \right)^{pol.} \right)  & \leq & 1 -\frac{2 (1 + x_1^2 + 2 x_1^4)}{(1 + x_1^2)^3} + \frac{
 2 (-x_1^2 + x_1^4) \cos[2 \phi]]}{(1 + x_1^2)^3}  \\
\end{eqnarray*}

Furthermore, the coefficient in front of the last term is non-positive definite, so the final term is maximized (in $\phi$) when $\phi = \frac{\pi}{2}$. As such, we have that
\begin{eqnarray*} \mathfrak{A}\left( v^{pol.}, \left(w^\sharp \right)^{pol.} \right)  & \leq &  1 + \frac{ -2 (1 + x_1^2 + 2 x_1^4)  + 2 (x_1^2 - x_1^4)}{(1 + x_1^2)^3} \\
& = &   1 + \frac{ -2 (1 +  3 x_1^4)}{(1 + x_1^2)^3}.
\end{eqnarray*}

To show that this is negative, we let $s = x_1^2$ and consider the quantity 
\begin{eqnarray*} (1+s)^3 \mathfrak{A}\left( v^{pol.}, \left(w^\sharp \right)^{pol.} \right)  & \leq &  (1+s)^3  -2 (1 +  3 x_1^4) \\
& = & (s-1)^3 \\
& < & 0 \textrm{ (for }s < 1).
\end{eqnarray*}

This shows that the anti-bisectional curvature is negative. To see that the metric is negatively cost-curved, note that when $v = w^\sharp$,

\[\mathfrak{A}\left( v^{pol.}, \left(w^\sharp \right)^{pol.} \right) =-2 +\frac{(-1 + s) (1 + s^2 - 2 s \cos[2 \theta])}{(1 + s)^3}  \]
which forces the expression to be strongly negative.

To show that the holomorphic sectional curvature is strongly negative, we note that the metric becomes asymptotically hyperbolic near the boundary, so it suffices to verify this property when $s$ is strictly smaller than $1$.
We then consider the non-polarized vector $\X = \partial z_1 + (a+ \sqrt{-1} b) \partial z_2$ and compute its holomorphic sectional curvature.
\begin{equation}
 (1 - s)^4 (1 + s) R(\V, \overline \V,\V, \overline \V)= -2 \left(
  \begin{aligned}
 &3 + 3 s + 9 s^2 + s^3 - 2 b^2 (1 - s)^3 \\
 &+ a^4 (1 - s)^2 (3 + s) + b^4 (1 - s)^2 (3 + s) \\
&+ 2 a^2 (1 - s) \left(-3 - 2 s - 3 s^2 + b^2 (-3 + 2 s + s^2) \right)
       \end{aligned} \right)
\end{equation}

We use Mathematica to find the maximum for this in terms of $a$ and $b$. Doing so we find that
\[ R(\V, \overline \V,\V, \overline \V) \leq -\frac{2(3+3s+9s^2+s^3)}{(s-1)^4(s+1)}. \]

This is strictly negative for $s <1$, so we find that the holomorphic sectional curvature is negative.

\subsection{Negative anti-bisectional curvature in higher dimensions}
In higher dimensions, we can again take advantage of symmetry to rotate so that all the coordinates except $x_1$ vanish, \begin{eqnarray*} v &=& \cos( \theta) \partial x_1 + \sin(\theta) \partial x_2, \textrm{ and } \\
w &=& \sin(\phi) dx_1 - \cos(\phi) \cos(\alpha) d x_2 - \cos(\phi) \sin(\alpha) d x_3. \end{eqnarray*}
In this case, again denoting $x_1^2 = s$ we can express the anti-bisectional curvature in the following way.

\begin{eqnarray*}
(1 + s)^3 \mathfrak{A}\left( v^{pol.}, \left(w^\sharp \right)^{pol.} \right) & = &  - (1 - s) (1 + 2 s + s^2) \cos[\theta]^2 \cos[\phi]^2 \cos[
     \alpha]^2  \\ 
     & & - (3 + 7 s + 5 s^2 + s^3) \cos[\phi]^2 \cos[\alpha]^2 \sin[ \theta]^2 \\
     & &- (3 + 3 s + 9 s^2 + s^3) \cos[\theta]^2 \sin[
     \phi]^2 \\
     & & - (1 - s) (1 - 2 s + s^2) \sin[\theta]^2 \sin[
     \phi]^2 \\
    & &  + (1 + s)^3 \cos[\alpha] \sin[2 \theta] \sin[
     2 \phi] \\
     & & - (1 - s) (1 + 2 s + s^2) \cos[\theta]^2 \cos[\phi]^2 \sin[
     \alpha]^2 \\
     & & - (1 - s) (1 + 2 s + s^2) \cos[\phi]^2 \sin[\theta]^2 \sin[\alpha]^2.
     \end{eqnarray*}
     We can simplify this expression using a routine but tedious calculation to find the following.
     \begin{eqnarray*}
 (1 + s)^3 \mathfrak{A}\left( v^{pol.}, \left(w^\sharp \right)^{pol.} \right)   & = &  - (1 - s) (1  + s)^2 \cos[\phi]^2 \\ 
  &  & - (1 - s)^3  \sin[ \phi]^2 \\
 & & -2 (1 + s)^3 \left( \cos[\theta] \sin[\phi] - \cos[\phi] \cos[\alpha] \sin[\theta] \right)^2. 
 \end{eqnarray*}
 
 Doing so, we find that
      \begin{eqnarray*}
 (1 + s)^3 \mathfrak{A}\left( v^{pol.}, \left(w^\sharp \right)^{pol.} \right)  & \leq & -(1-s)^3.
     \end{eqnarray*}
This shows that the metric has negative anti-bisectional curvature. We can then use the same calculation as in the two-dimensional case to show that the holomorphic sectional curvature is strongly negative.

\section{The polarized holomorphic sectional curvature and $\mathfrak{O}$ control the Riemann curvature tensor for metrics with negative anti-bisectional curvature}
\label{PHSCandO}

Here we show that for a metric with negative anti-bisectional curvature, we can control the Riemann curvature tensor in terms of the polarized holomorphic sectional curvatures and the orthogonal anti-bisectional trace curvature.

As before, we work in a polarized unitary frame. Let $s$ be a free parameter and consider
\begin{eqnarray*}
\mathfrak{A}(\E^i, \E^j+t \E^k) & = & \mathfrak{A}(\E^i, \E^j) + 2 s R(\E^i, \overline \E^j,\E^i, \overline \E^k) + s^2 \mathfrak{A}(\E^i, \E^k).
\end{eqnarray*}
The left hand side is non-positive. By our assumptions, we have bounds on the first and third terms, which allows us to obtain two sided bounds on the middle term as well.

Now we consider the quantity
\begin{eqnarray*}
H(\E^i + \E^j) & = & H(\E^i) + 4  R(\E^i, \overline \E^j,\E^i, \overline \E^i) + 4 R(\E^i, \overline \E^i,\E^j, \overline \E^j)  \\
& &+ 2 \mathfrak{A}(\E^i,  \E^j) + 4  R(\E^i, \overline \E^j,\E^j, \overline \E^j) + H(\E^i).
\end{eqnarray*}

Again, the left hand side is negative and bounded from below. From the previous estimate (setting $i=j$) and our assumption, we have two sided bounds on all the terms except for $4 R(\E^i, \overline \E^i,\E^j, \overline \E^j)$. Rearranging the expression, we obtain two-sided bounds on this term.

Now we consider the expression
\begin{eqnarray*}
\mathfrak{A}(\E^i +t \E^j, \E^i+s \E^k) & = & 2 ts R(\E^i, \overline \E^i, \E^j, \overline \E^k) + \textrm{bounded terms},
\end{eqnarray*}
where both $s$ and $t$ are free parameters.
The left hand side is negative, so we derive a two-sided bound on the first term by choosing the parameters appropriately.

Finally, we consider
\begin{eqnarray*}
\mathfrak{A}(\E^i +t \E^k, \E^j+s \E^\ell) & = & 2 ts R(\E^i, \overline \E^j, \E^k, \overline \E^\ell) + \textrm{terms with a repeated index}.
\end{eqnarray*}
Since all the terms with a repeated index have been bounded and the left hand side is non-positive, we can choose $t$ and $s$ to obtain two-sided bounds on the first term on the right hand side. This gives a bound on the curvature in our polarized unitary frame, which shows that the curvature is bounded.


\end{appendix}


\begin{thebibliography}{99}
\bibitem{Berger} Berger, M. (1961). Sur quelques vari\`et\`es d'Einstein compactes. Annali di Matematica Pura ed Applicata, 53(1), 89-95.
\bibitem{Brendle} Brendle, S. (2018). Ricci flow with surgery in higher dimensions. Annals of Mathematics, 263-299.
\bibitem{Boucksom} Boucksom, S., \& Eyssidieux, P. (2013). An introduction to the K\"ahler-Ricci flow (Vol. 2086). V. Guedj (Ed.). Cham: Springer.
\bibitem{YB} Brenier, Y. (1987) D\'ecomposition polaire et r\'earrangement monotone des champs de vecteurs. C.R.
Acad. Sci. Paris S\'er. I Math., 305, 805–808.
\bibitem{LC} Caffarelli, L. A. (1992). The regularity of mappings with a convex potential. Journal of the American Mathematical Society, 5(1), 99-104.
\bibitem{Calabi} Calabi, E. (1975). A construction of nonhomogeneous Einstein metrics. In Proc. of Symp. in Pure Mathematics (Vol. 27, pp. 17-24). Amer. Math. Soc..
\bibitem{CaoHarnack} Cao, H. D. (1992). On Harnack's inequalities for the K\"ahler-Ricci flow. Inventiones mathematicae, 109(1), 247-263.
\bibitem{CaoDeformation} Cao, H. D. (1986). Deformation of Kahler metrics to Kahler-Einstein metrics on compact Kahler manifolds.
\bibitem{Cao} Cao, H. D. (2013). The Kähler–Ricci Flow on Fano Manifolds. In An Introduction to the Kähler-Ricci Flow (pp. 239-297). Springer, Cham.
\bibitem{Chen} Chen, B. L. (2009). Strong uniqueness of the Ricci flow. Journal of differential geometry, 82(2), 363-382.
\bibitem{ChenZhu} Chen, B. L., \& Zhu, X. P. (2006). Uniqueness of the Ricci flow on complete noncompact manifolds. Journal of Differential Geometry, 74(1), 119-154.
\bibitem{ChengYau} Cheng, S. Y., \& Yau, S. T. (1977). On the regularity of the monge‐ampère equation det$(\partial^2 u/ \partial x_i \partial x_j)= f (x, u)$. Communications on Pure and Applied Mathematics, 30(1), 41-68.
\bibitem{Delanoe} Delano\"e, P. (1991). Classical solvability in dimension two of the second boundary-value problem associated with the Monge-Ampere operator. In Annales de l'Institut Henri Poincare (C) Non Linear Analysis (Vol. 8, No. 5, pp. 443-457). Elsevier Masson.
\bibitem{DePhillipisFigalli} De Philippis, G., \& Figalli, A. (2014). The Monge-Amp\`ere equation and its link to optimal transportation. Bulletin of the American Mathematical Society, 51(4), 527-580.
\bibitem{Deturck} DeTurck, D. M. (1981). Existence of metrics with prescribed Ricci curvature: local theory. Inventiones mathematicae, 65(2), 179-207.
\bibitem{Dombrowski} Dombrowski, P. (1962). On the geometry of the tangent bundle. Journal f\"ur Mathematik. Bd, 210(1/2), 10.
\bibitem{FigalliKimMcCann} Figalli, A., Kim, Y. H., \& McCann, R. J. (2013). H\"older continuity and injectivity of optimal maps. Archive for Rational Mechanics and Analysis, 209(3), 747-795.
\bibitem{GOT} Gangbo, W., \& McCann, R. J. (1996). The geometry of optimal transportation. Acta Mathematica, 177(2), 113-161.
\bibitem{Folland} Folland, G. B. (2001). How to integrate a polynomial over a sphere. The American Mathematical Monthly, 108(5), 446-448.
\bibitem{GuillenKitagawa} Guillen, N., \& Kitagawa, J. (2015). On the local geometry of maps with c-convex potentials. Calculus of Variations and Partial Differential Equations, 52(1-2), 345-387.
\bibitem{GuZhang} Gu, H., \& Zhang, Z. (2010). An extension of Mok’s theorem on the generalized Frankel conjecture. Science China Mathematics, 53(5), 1253-1264.
\bibitem{HamiltonLiYauHarnack} Hamilton, R. S. (1993). The Harnack estimate for the Ricci flow. Journal of Differential Geometry, 37(1), 225-243.
\bibitem{HamiltonFourmanifolds} Hamilton, R. S. (1986). Four-manifolds with positive curvature operator. Journal of Differential Geometry, 24(2), 153-179.
\bibitem{Hamiltonthreemanifolds} Hamilton, R. S. (1982). Three-manifolds with positive Ricci curvature. Journal of Differential Geometry, 17(2), 255-306.
\bibitem{MTWNotebook} Khan, G. The Curvature of the Sasaki metric. Mathematica Notebook. Available at 
\href{https://www.wolframcloud.com/obj/0e49f5e5-bc50-4954-83cd-24db2f87206b}{\tt https://www.wolframcloud.com/obj/0e49f5e5-bc50-4954-83cd-24db2f87206b}
\bibitem{KhanZhang} Khan, G., \& Zhang, J. (2020). The K\"ahler geometry of certain optimal transport problems. Pure and Applied Analysis, 2(2), 397-426.
\bibitem{KhanZhangZheng} Khan, G., Zhang, J., \& Zheng, F. (2020). The geometry of positively curved K\" ahler metrics on tube domains. arXiv preprint arXiv:2001.06155.
\bibitem{KimMcCann} Kim, Y. H., \& McCann, R. (2010). Continuity, curvature, and the general covariance of optimal transportation. Journal of the European Mathematical Society, 12(4), 1009-1040.
\bibitem{KimMcCannWarren} Kim, Y. H., McCann, R. J., \& Warren, M. (2010). Pseudo-Riemannian geometry calibrates optimal transportation. Mathematical Research Letters, 17(6), 1183-1197.
\bibitem{KitagawaParabolic} Kitagawa, J. (2012). A parabolic flow toward solutions of the optimal transportation problem on domains with boundary. Journal f\"ur die reine und angewandte Mathematik, 2012(672), 127-160.
\bibitem{Klembeck} Klembeck, P. F. (1978). K\"ahler metrics of negative curvature, the Bergmann metric near the boundary, and the Kobayashi metric on smooth bounded strictly pseudoconvex sets. Indiana University Mathematics Journal, 27(2), 275-282.
\bibitem{LiYau} Li, P., \& Yau, S. T. (1986). On the parabolic kernel of the Schr\"odinger operator. Acta Mathematica, 156, 153-201.
\bibitem{Liu} Liu, G. (2019). On Yau’s uniformization conjecture. Cambridge Journal of Mathematics, 7(1), 33-70.
\bibitem{Loeper} Loeper, G. (2009). On the regularity of solutions of optimal transportation problems. Acta mathematica, 202(2), 241-283.
\bibitem{MTW} Ma, X. N., Trudinger, N. S., \& Wang, X. J. (2005). Regularity of potential functions of the optimal transportation problem. Archive for rational mechanics and analysis, 177(2), 151-183.
\bibitem{McCannTopping} McCann, R. J., \& Topping, P. M. (2010). Ricci flow, entropy and optimal transportation. American Journal of Mathematics, 132(3), 711-730.
\bibitem{MirghafouriMalek} Mirghafouri, M., \& Malek, F. (2017). Long-time existence of a geometric flow on closed Hessian manifolds. Journal of Geometry and Physics, 119, 54-65.
\bibitem{Mok} Mok, N. (1988). The uniformization theorem for compact K\"ahler manifolds of nonnegative holomorphic bisectional curvature. Journal of Differential Geometry, 27(2), 179-214.
\bibitem{Molitor} Molitor, M. (2014). Gaussian distributions, Jacobi group, and Siegel-Jacobi space. Journal of Mathematical Physics, 55(12), 122102.
\bibitem{Monge}  Monge, G. (1781). M\'emoire sur la th\'eorie des d\'eblais et des remblais. Histoire de l'Acad\'emie Royale des Sciences de Paris.
\bibitem{Nguyen} Nguyen, H. T. (2010). Isotropic curvature and the Ricci flow. International Mathematics Research Notices, 2010(3), 536-558.
\bibitem{PalWong} Pal, S., \& Wong, T. K. L. (2018). Exponentially concave functions and a new information geometry. The Annals of Probability, 46(2), 1070-1113.
\bibitem{Perelmanentropy} Perelman, G. (2002). The entropy formula for the Ricci flow and its geometric applications. arXiv preprint math/0211159.
\bibitem{PuechmorelTo} Puechmorel, S., \& T\^o, T. D. (2020). Convergence of the Hesse-Koszul flow on compact Hessian manifolds. arXiv preprint arXiv:2001.02940.
\bibitem{Royden} Royden, H. L. (1980). The Ahlfors-Schwarz lemma in several complex variables. Commentarii Mathematici Helvetici, 55(1), 547-558.
\bibitem{Santambrogio} Santambrogio, F. (2015). Optimal transport for applied mathematicians. Birk\"auser, NY, 55(58-63), 94.
\bibitem{Satoh} Satoh, H. (2007). Almost Hermitian structures on tangent bundles. In Workshop on Diff. Geom (Vol. 11, pp. 105-118).
\bibitem{Shi} Shi, W. X. (1997). Ricci flow and the uniformization on complete noncompact K\"ahler manifolds. Journal of Differential Geometry, 45(1), 94-220.
\bibitem{Shima} Shima, H. (1981). Hessian manifolds and convexity. In Manifolds and Lie groups (pp. 385-392). Birkh\"auser, Boston, MA.
\bibitem{Shimabook} Shima, H. (2007). The geometry of Hessian structures. World Scientific.
\bibitem{SiuYang} Siu, Y. T., \& Yang, P. (1981). Compact K\"ahler-Einstein surfaces of nonpositive bisectional curvature. Inventiones mathematicae, 64(3), 471-487.
\bibitem{SongTian} Song, J., \& Tian, G. (2017). The K\"ahler–Ricci flow through singularities. Inventiones mathematicae, 207(2), 519-595.
\bibitem{Tong} Tong, F. (2018). The K\" ahler-Ricci flow on manifolds with negative holomorphic curvature. arXiv preprint arXiv:1805.03562.
\bibitem{Topping} Topping, P. (2009). $\mathcal{L}$-optimal transportation for Ricci flow. Journal f\"ur die reine und angewandte Mathematik, 2009(636), 93-122.
\bibitem{TW} Trudinger, N. S., \& Wang, X. J. (2009). On the second boundary value problem for Monge-Ampere type equations and optimal transportation. Annali della Scuola Normale Superiore di Pisa-Classe di Scienze-Serie IV, 8(1), 143.
\bibitem{Urbas} Urbas, J. (1997). On the second boundary value problem for equations of Monge-Ampere type. Journal fur die Reine und Angewandte Mathematik, 487, 115-124.
\bibitem{vanCoevering} van Coevering, C. (2012). K\"ahler–Einstein metrics on strictly pseudoconvex domains. Annals of Global Analysis and Geometry, 42(3), 287-315.
\bibitem{OTON} Villani, C. (2008). Optimal transport: old and new (Vol. 338). Springer Science \& Business Media.
\bibitem{VillaniStability} Villani, C. (2008). Stability of a 4th-order curvature condition arising in optimal transport theory. Journal of Functional Analysis, 255(9), 2683-2708.
\bibitem{Vinberg} Vinberg, E. B. (1967). Theory of homogeneous convex cones. Trans. Moscow Math. Soc., 12, 303-368.
\bibitem{Wilking} Wilking, B. (2013). A Lie algebraic approach to Ricci flow invariant curvature conditions and Harnack inequalities. Journal f\"ur die reine und angewandte Mathematik, 2013(679), 223-247.
\bibitem{WuYau} Wu, D., \& Yau, S. T. (2020). Invariant metrics on negatively pinched complete K\"ahler manifolds. Journal of the American Mathematical Society, 33(1), 103-133.
\bibitem{YauCalabi} Yau, S. T. (1977). Calabi's conjecture and some new results in algebraic geometry. Proceedings of the National Academy of Sciences, 74(5), 1798-1799.
\bibitem{YauRicci} Yau, S. T. (1978). On the ricci curvature of a compact K\"ahler manifold and the complex Monge‐Amp\`ere equation, I. Communications on pure and applied mathematics, 31(3), 339-411.
\bibitem{ZhangKhan} Zhang, J., \& Khan, G. (2020). Statistical mirror symmetry. Differential Geometry and its Applications, 73, 101678.
\end{thebibliography}
\end{document}